\documentclass[a4paper, reqno, 11pt]{amsart}
\usepackage[usenames,dvipsnames]{color}
\usepackage{amsthm,amsfonts,amssymb,amsmath,amsxtra}
\usepackage[all]{xy}
\SelectTips{cm}{}
\usepackage{xr-hyper}
\usepackage[colorlinks=
   citecolor=Black,
   linkcolor=Red,
   urlcolor=Blue]{hyperref}
\usepackage{verbatim}
\usepackage{caption}

\usepackage{lineno}

\usepackage[margin=1.25in]{geometry}
\usepackage{mathrsfs}

\newcommand{\remind}[1]{{\bf ** #1 **}}

\RequirePackage{xspace}
\RequirePackage{etoolbox}
\RequirePackage{varwidth}
\RequirePackage{enumitem}
\RequirePackage{tensor}
\RequirePackage{mathtools}
\RequirePackage{longtable}
\RequirePackage{multirow}

\def\lup{\rightharpoonup}
\def\ge{\geqslant}
\def\le{\leqslant}
\def\a{\alpha}
\def\b{\beta}

\def\G{\Gamma}
\def\d{\delta}

\def\e{\epsilon}

\def\o{\omega}

\def\s{\sigma}
\def\t{\tau}

\def\k{\kappa}
\def\l{\lambda}

\def\i{^{-1}}

\def\<{\langle}
\def\>{\rangle}
\def\bq{\mathbf q}
\def\bH{\mathbf H}

\newcommand{\spade}{{\spadesuit}}

\newcommand{\bG}{\mathbf G}

\newcommand{\bM}{\mathbf M}

\newcommand{\bJ}{\mathbf J}

\newcommand{\BA}{\ensuremath{\mathbb {A}}\xspace}

\newcommand{\BC}{\ensuremath{\mathbb {C}}\xspace}

\newcommand{\BF}{\ensuremath{\mathbb {F}}\xspace}
\newcommand{{\BG}}{\ensuremath{\mathbb {G}}\xspace}

\newcommand{{\BK}}{\ensuremath{\mathbb {K}}\xspace}

\newcommand{\BN}{\ensuremath{\mathbb {N}}\xspace}

\newcommand{\BQ}{\ensuremath{\mathbb {Q}}\xspace}
\newcommand{\BR}{\ensuremath{\mathbb {R}}\xspace}
\newcommand{\BS}{\ensuremath{\mathbb {S}}\xspace}

\newcommand{\BZ}{\ensuremath{\mathbb {Z}}\xspace}

\newcommand{\CA}{\ensuremath{\mathcal {A}}\xspace}

\newcommand{\CI}{\ensuremath{\mathcal {I}}\xspace}

\newcommand{\CO}{\ensuremath{\mathcal {O}}\xspace}

\newcommand{\CT}{\ensuremath{\mathcal {T}}\xspace}

\newcommand{\Ad}{{\mathrm{Ad}}}
\newcommand{\ad}{{\mathrm{ad}}}

\newcommand{\de}{{\mathrm{def}}}

\DeclareMathOperator{\Gal}{Gal}
\newcommand{\GL}{\mathrm{GL}}

\newcommand{\indec}{\mathrm{indec}}
\newcommand{\irr}{\mathrm{irr}}
\newcommand{\leng}{\mathrm{length}}

\newcommand{\id}{\ensuremath{\mathrm{id}}\xspace}

\DeclareMathOperator{\rank}{rank}

\def\tW{\tilde W}
\def\tS{\tilde \BS}
\def\kk{\mathbf k}
\DeclareMathOperator{\supp}{supp}
\def\leng{{\rm length}}

\newtheorem{theorem}{Theorem}
\newtheorem{proposition}[theorem]{Proposition}
\newtheorem{lemma}[theorem]{Lemma}

\newtheorem{corollary}[theorem]{Corollary}

\theoremstyle{definition}

\newtheorem{remark}[theorem]{Remark}

\numberwithin{equation}{section}
\numberwithin{theorem}{section}

\setitemize[0]{leftmargin=*,itemsep=\the\smallskipamount}
\setenumerate[0]{leftmargin=*,itemsep=\the\smallskipamount}

\renewcommand{\to}{%
   \ifbool{@display}{\longrightarrow}{\rightarrow}%
   }
\let\shortmapsto\mapsto
\renewcommand{\mapsto}{%
   \ifbool{@display}{\longmapsto}{\shortmapsto}%
   }
\newlength{\olen}
\newlength{\ulen}
\newlength{\xlen}
\newcommand{\xra}[2][]{%
   \ifbool{@display}%
      {\settowidth{\olen}{$\overset{#2}{\longrightarrow}$}%
       \settowidth{\ulen}{$\underset{#1}{\longrightarrow}$}%
       \settowidth{\xlen}{$\xrightarrow[#1]{#2}$}%
       \ifdimgreater{\olen}{\xlen}%
          {\underset{#1}{\overset{#2}{\longrightarrow}}}%
          {\ifdimgreater{\ulen}{\xlen}%
             {\underset{#1}{\overset{#2}{\longrightarrow}}}
             {\xrightarrow[#1]{#2}}}}%
      {\xrightarrow[#1]{#2}}
   }
\makeatother
\newcommand{\xyra}[2][]{%
   \settowidth{\xlen}{$\xrightarrow[#1]{#2}$}%
   \ifbool{@display}%
      {\settowidth{\olen}{$\overset{#2}{\longrightarrow}$}%
       \settowidth{\ulen}{$\underset{#1}{\longrightarrow}$}%
       \ifdimgreater{\olen}{\xlen}%
          {\mathrel{\xymatrix@M=.12ex@C=3.2ex{\ar[r]^-{#2}_-{#1} &}}}%
          {\ifdimgreater{\ulen}{\xlen}%
             {\mathrel{\xymatrix@M=.12ex@C=3.2ex{\ar[r]^-{#2}_-{#1} &}}}
             {\mathrel{\xymatrix@M=.12ex@C=\the\xlen{\ar[r]^-{#2}_-{#1} &}}}}}%
      {\mathrel{\xymatrix@M=.12ex@C=\the\xlen{\ar[r]^-{#2}_-{#1} &}}}%
   }
\makeatletter
\newcommand{\xla}[2][]{%
   \ifbool{@display}%
      {\settowidth{\olen}{$\overset{#2}{\longleftarrow}$}%
       \settowidth{\ulen}{$\underset{#1}{\longleftarrow}$}%
       \settowidth{\xlen}{$\xleftarrow[#1]{#2}$}%
       \ifdimgreater{\olen}{\xlen}%
          {\underset{#1}{\overset{#2}{\longleftarrow}}}%
          {\ifdimgreater{\ulen}{\xlen}%
             {\underset{#1}{\overset{#2}{\longleftarrow}}}
             {\xleftarrow[#1]{#2}}}}%
      {\xleftarrow[#1]{#2}}
   }
\newcommand{\isoarrow}{%
   \ifbool{@display}{\overset{\sim}{\longrightarrow}}{\xrightarrow\sim}%
   }
   
\begin{document}

\title[Affine Deligne--Lusztig varieties with finite Coxeter parts]{Affine Deligne--Lusztig varieties with finite Coxeter parts}
\author[Xuhua He]{Xuhua He}
\address{The Institute of Mathematical Sciences and Department of Mathematics, The Chinese University of Hong Kong, Shatin, N.T., Hong Kong SAR, China}
\email{xuhuahe@math.cuhk.edu.hk}

\author[Sian Nie]{Sian Nie}
\address{Academy of Mathematics and Systems Science, Chinese Academy of Sciences, Beijing 100190, China, and, School of Mathematical Sciences, University of Chinese Academy of Sciences, Chinese Academy of Sciences, Beijing 100049, China}
\email{niesian@amss.ac.cn}

\author[Qingchao Yu]{Qingchao Yu}
\address{Beijing International Center for Mathematical Research, Beijing University, Haidian, Beijing, China}
\email{yuqingchao@bicmr.pku.edu.cn}
\thanks{}


\begin{abstract}
In this paper, we study  affine Deligne--Lusztig varieties $X_w(b)$ when the finite part of the element $w$ in the Iwahori--Weyl group is a partial $\s$-Coxeter element. We show that such $w$ is a cordial element and $X_w(b) \neq \emptyset$ if and only if $b$ satisfies a certain Hodge--Newton indecomposability condition. The main result of this paper is that for such $w$ and $b$,  $X_w(b)$ has a simple geometric structure: the $\s$-centralizer of $b$ acts transitively on the set of irreducible components of $X_w(b)$; and each irreducible component is an iterated fibration over a classical Deligne--Lusztig variety of Coxeter type, and the iterated fibers are either $\BA^1$ or $\BG_m$. 
\end{abstract}

\maketitle

\tableofcontents

\section{Introduction}
\subsection{Classical/affine Deligne--Lusztig varieties}
The classical Deligne--Lusztig varieties were introduced by Deligne and Lusztig in \cite{DL}. They play a crucial role in the representation theory of finite reductive groups. They are defined for a connected reductive group $\bG$ over a finite field $\BF_q$. For any element $w$ in the (finite) Weyl group of $\bG(\bar \BF_q)$, the corresponding Deligne--Lusztig variety $X_w$ is a certain locally closed subvariety of the flag variety of $\bG(\bar \BF_q)$, and it admits a natural action of the finite reductive group $\bG(\BF_q)$. It is known that

\smallskip
(a) {\it The classical Deligne--Lusztig variety $X_w$ is smooth and of dimension equal to the length of $w$, and the finite reductive group $\bG(\BF_q)$ acts transitively on the set of irreducible components of $X_w$.} \\
\smallskip

Affine Deligne--Lusztig varieties were introduced by Rapoport in \cite{Ra} as the affine analog of  classical Deligne--Lusztig varieties. They serve as  group-theoretic models for Shimura varieties and shtukas. They are defined for a connected reductive group $\bG$ over a non-archimedean local field $F$. Let $\breve F$ be the completion of the maximal unramified extension of $F$. For any element $w$ in the Iwahori--Weyl group $\tW$ of $\bG(\breve F)$ and any element $b \in \bG(\breve F)$, the corresponding affine Deligne--Lusztig variety $X_w(b)$ is a certain locally closed subscheme of finite type in the affine flag variety of $\bG(\breve F)$, and it admits a natural action of the $\s$-centralizer $\bJ_b(F)$ of $b$. Unlike  classical Deligne--Lusztig varieties, which have the nice geometric structure described in (a), the geometric structures of  affine Deligne--Lusztig varieties are very complicated: 
\begin{itemize}
    \item For many pairs $(w, b)$, $X_w(b)$ are empty.
    \item Even if $X_w(b)$ is nonempty, it is not equi-dimensional in general, and it is very difficult to determine its dimension.
    \item In general, the group $\bJ_b(F)$ does not acts transitively on the set of irreducible components of $X_w(b)$, and very little is known about this  $\bJ_b(F)$-action. 
    \item The irreducible components of $X_w(b)$, in general, have a very complicated geometric structure. 
\end{itemize}

We refer to the survey article \cite{He-ICM} and the paper \cite{He-Pi} for recent developments regarding the nonemptiness pattern and the dimension formula for $X_w(b)$. 

\subsection{Main result} 
In \cite{MV}, Mili\'cevi\'c and Viehmann introduced the notion of cordial elements. The geometry of the affine Deligne--Lusztig varieties associated with  cordial elements is ``well-behaved'' in the following sense. If $w$ is a cordial element, then the elements $b$ with $X_w(b) \neq \emptyset$ form a saturated set in the sense of \cite[Theorem 1.1]{MV}, and for any such $b$, there is a simple dimension formula for $X_w(b)$. Moreover, $X_w(b)$ is equi-dimensional. Schremmer gave a classification of the cordial elements  in his recent work \cite{Sch}.

However, even for a cordial element $w$, very little is known about the $\bJ_b(F)$-orbits on the set of irreducible components of $X_w(b)$ or about the geometric structure of the irreducible components of $X_w(b)$. 

On the other hand, by \cite{He14}, there is a family of elements in the Iwahori--Weyl group whose associated affine Deligne--Lusztig varieties have very simple geometric structures. We denote by $\s$ the Frobenius morphism on $\bG(\breve F)$ and the induced group automorphism on the Iwahori--Weyl group $\tW$. Suppose that $w$ is a minimal length element in its $\s$-conjugacy class of $\tW$; then there exists a unique $\s$-conjugacy class $[b]$ with $X_w(b) \neq \emptyset$. In this case, there exist a parahoric subgroup $P$ of $\bJ_b(F)$ and a classical Deligne--Lusztig variety $X$ (associated with the reductive quotient of $P$) such that $X_w(b) \cong \bJ_b(F) \times^{P} X$. Such a simple geometric structure has been used in the study of certain Shimura varieties with simple geometric structure (see \cite{GH}, \cite{GHN-full}, and \cite{GHN2}). However, these minimal length elements $w$ form only a tiny fraction of the whole Iwahori--Weyl group, and such a simple geometric structure only occurs in a few cases. 

In this paper, we will focus on another family of elements in $\tW$. For any $w \in \tW$, we define its finite part to be the image of $w$ under the map $\eta_\s: \tW \to W$ (see \S\ref{sec:cordial}). The main result of this paper is that if the finite part of $w$ is a $\s$-Coxeter element of $W$, then the associated affine Deligne--Lusztig variety $X_w(b)$ for any $b$ has a simple geometric structure. 

\begin{theorem}[See Theorem~\ref{main}]
Let $w \in \tW$ such that $\eta_\s(w)$ is a $\s$-Coxeter element of $W$. Then
\begin{enumerate}
\item  $w$ is a cordial element; 
\item  $X_w(b) \neq \emptyset$ if and only if $b$ satisfies a certain Hodge--Newton indecomposability condition; 
\item  For any $b$ with $X_w(b) \neq \emptyset$, there exists a parahoric subgroup $P$ of $\bJ_b(F)$ and a classical Deligne--Lusztig variety $X$ of Coxeter type, and an iterated fibration $Y \to X$ whose iterated fibers are either $\BA^1$ or $\BG_m$ such that $X_w(b) \cong \bJ_b(F) \times^P Y$. 
\end{enumerate}
\end{theorem}

We refer to \S\ref{sec:notation} for the precise statement and the definition of the notions used here. The special case of Part~(3) where $\bG = \GL_n$, $b$ is basic and $w$ is certain element with finite Coxeter part was studied by Shimada \cite{Shi}. 

\subsection{Strategy}
One major tool used in the study of affine Deligne--Lusztig varieties is the Deligne--Lusztig reduction method \cite{He14}. Based on the Deligne--Lusztig reduction, a close relationship between  affine Deligne--Lusztig varieties and the class polynomials of  affine Hecke algebras was established in \cite{He14}. One remarkable property of these class polynomials is that they are polynomials in $(\bq-1)$ with non-negative integral coefficients. With each element $w \in \tW$, we may associate a reduction tree, which encodes the information on the reduction steps and determines the class polynomials associated with $w$. However, obtaining an explicit description of the reduction trees is quite challenging. 

Another key ingredient in this paper is the Chen--Zhu conjecture. This conjecture predicts the $\bJ_b(F)$-action on the top-dimensional irreducible components of  affine Deligne--Lusztig varieties in the affine Grassmannian. This conjecture was verified recently in \cite{Nie}, \cite{ZZ}, and \cite{HZZ}. Part of the Chen--Zhu conjecture predicts the isotropy group for the $\bJ_b(F)$-action, which gives some information about the end points of the reduction trees. 

Combining the above two ingredients, in \S\ref{sec:end}, we show that the end points of each reduction tree for $w$ with finite Coxeter part must be certain $\s$-Coxeter elements, and
each path, which corresponds to a $\s$-conjugacy class $[b]$ of $b$, in a reduction tree provides a $\bJ_b(F)$-orbit of irreducible components of $X_w(b)$. It remains to show that for any $b$, there is at most one path in the reduction tree that corresponds to $[b]$ (i.e., the ``multiplicity one'' result). For the (unique) maximal $\s$-conjugacy class $[b]$ with $X_w(b) \neq \emptyset$, this ``multiplicity one'' result is obvious. For the basic $\s$-conjugacy class $[b]$, one may deduce the ``multiplicity one'' result by showing that any path corresponding to $[b]$ is of a unique type. See \S\ref{extreme-case}.

It is more challenging to determine the numbers of reduction paths for other $\s$-conjugacy classes in a reduction tree. In this paper, we use the following indirect approach to establish the ``multiplicity one'' result. We first interpret the class polynomials as the number of  rational points for certain admissible subsets. We then use the positivity property of the class polynomials to show that the ``multiplicity one'' result for all $b$ is equivalent to the following single combinatorial identity: \[\tag{*} \sum_{[b] \in B(\bG, \mu)_{\indec}} (\bq-1)^? \bq^{-??}=1.\] Here $B(\bG, \mu)_{\indec}$ is the set of all Hodge--Newton indecomposable $\s$-conjugacy classes, and the powers ``?'' and ``??'' are certain non-negative integers determined by $w$ and $b$. We refer to \S\ref{sec:HN} and \S\ref{sec:5.3} for the precise definitions.

Verifying the combinatorial identity (*) is the most technical part of this paper and is done in \S\ref{sec:red}. We first establish natural bijections between the sets $B(\bG, \mu)_{\indec}$ for various pairs $(\bG, \mu)$, which is of independent interest. In combination with other techniques, we reduce the verification of (*) to the case for simply laced, $\breve F$-simple and split groups and for fundamental coweights. 

For classical groups, we may further reduce to the case where $\mu$ is minuscule. In this case, the ``multiplicity one'' result follows from the Chen--Zhu conjecture. For exceptional groups, we use a computer to verify (*). The most complicated case for the exceptional group is  $(E_8, \omega^\vee_4)$. In this case, the left-hand side of (*) involves a summation of $729$ terms. It is also worth mentioning that in the case $(A_{n-1}, \omega^\vee_i)$, we may write the identity (*) explicitly as 
\begin{multline*}
\sum_{\substack{k \ge 1, 1>\frac{a_1}{b_1}>\cdots>\frac{a_k}{b_k}>0; \\ a_i+\cdots+a_k=i, b_1+\cdots+b_k=n}} (\bq-1)^{k-1} \bq^{k-1-\frac{\sum_{1 \le l_1<l_2\le k}(a_{l_1} b_{l_2}-a_{l_2} b_{l_1})+\sum_{1 \le l \le k} \gcd(a_l, b_l)}{2}}\\
=\bq^{\frac{i(n-i)-n}{2}}.
\end{multline*}
We do not know if there is a purely combinatorial proof of this identity.

\smallskip

\noindent {\bf Acknowledgments: } XH is partially supported by Hong Kong RGC Grant 14300220, by funds connected with the Choh-Ming Chair at CUHK, and by the Xplorer prize.

\section{Preliminaries}\label{sec:notation}

\subsection{Reductive groups} Let $F$ be a non-archimedean local field with residue field $\BF_q$ and let $\breve F$ be the completion of the maximal unramified extension of $F$. We write $\Gamma$ for $\Gal(\overline F/F)$, and  $\Gamma_0$ for the inertia subgroup of $\Gamma$.

Let $\bG$ be a quasi-split connected reductive group over $F$. We set $\breve G=\bG(\breve F)$. Let $\s$ be the Frobenius morphism of $\breve F$ over $F$. We use the same symbol $\s$ for the induced Frobenius morphism on $\breve G$. Let $S$ be a maximal $\breve F$-split torus of $\bG$ defined over $F$, which contains a maximal $F$-split torus. Let $T$ be the centralizer of $S$ in $\bG$. Then $T$ is a maximal torus. We denote by $N$ the normalizer of $T$ in $\bG$. Let $W=N(\breve F)/T(\breve F)$ be the {\it relative Weyl group}. We fix a $\s$-stable Iwahori subgroup $\breve \CI$ of $\breve G$. Let $\tW= N(\breve F)/T(\breve F) \cap \breve \CI$ be the \emph{Iwahori--Weyl group}. The action $\s$ on $\breve G$ induces a natural action on $\tW$ and $W$, which we still denote by $\s$. For any $w \in \tW$, we choose a representative $\dot w$ in $N(\breve F)$. We have the splitting $$\tW=X_*(T)_{\G_0} \rtimes W=\{t^{\l} w; \l \in X_*(T)_{\G_0}, w \in W\}.$$ Here $X_*(T)_{\G_0}$ denotes the $\G_0$-coinvariants of $X_*(T)$.

Since $\bG$ is quasi-split over $F$, $\s$ acts naturally on $X_*(T)_{\G_0}$ and on $W$. We denote by $\ell$ the length function on $\tW$ and on $W$, and by $\le$ the Bruhat order on $\tW$ and on $W$. Let $\tilde \BS$ be the index set of simple reflections in $\tW$ and let $\BS \subset \tilde \BS$ be the index set of simple reflections in $W$. In other words, the set of simple reflections in $\tW$ is $\{s_i\mid i\in \tilde \BS\}$. 

For any $w \in W$, we denote by $\supp(w)$ the set of $i$ such that $s_i$ occurs in some (or, equivalently, any) reduced expressions of $w$, and we set $\supp_\s(w)=\cup_{l \in \BN} \s^l(\supp(w))$. 

An element $c \in W$ is called a {\it (full) $\s$-Coxeter element} if it is a product of simple reflections, one from each $\s$-orbit of $\BS$. An element $c \in W$ is called a {\it partial $\s$-Coxeter element} if it is a product of simple reflections, at most one from each $\s$-orbit of $\BS$. 

Let $\Phi$ be the reduced root system underlying the relative root system of $\bG$ over $\breve F$ (the {\it \'echelonnage root system}). For any $i \in \BS$, we denote by $\a_i$ and $\a_i^\vee$ the corresponding (relative) simple root and simple coroot respectively.

\subsection{The $\s$-conjugacy classes of $\breve G$}\label{sec:2.2}

The $\s$-conjugation action on $\breve G$ is defined by $g \cdot_\s g'=g g' \s(g) \i$. For $b \in \breve G$, we denote by $[b]$ the $\s$-conjugacy class of $b$. Let $B(\bG)$ be the set of $\s$-conjugacy classes on $\breve G$. The classification of the $\s$-conjugacy classes is due to Kottwitz in \cite{kottwitz-isoI} and \cite{kottwitz-isoII}. Any $\s$-conjugacy class $[b]$ is determined by two invariants: 
\begin{itemize}
	\item the element $\k([b]) \in \pi_1(\bG)_{\s}$; 	
	\item the Newton point $\nu_b \in \big((X_*(T)_{\Gamma_0, \BQ})^+\big)^{\langle\sigma\rangle}$. 
\end{itemize}

Here $\pi_1(\bG)_{\s}$ denotes the $\s$-coinvariants of $\pi_1(\bG)$, and
$(X_*(T)_{\Gamma_0, \BQ})^+$ denotes the intersection of  $X_*(T)_{\Gamma_0}\otimes \BQ=X_*(T)^{\Gamma_0}\otimes \BQ$ with the set $X_*(T)_\BQ^+$ of dominant elements in $X_*(T)_\BQ$. Denote $V=X_*(T)_{\Gamma_0}\otimes \BR$. For any $v \in V $, denote $I(v) = \{i\in\BS; \<v,\a_i \>=0\}$. Here $\< ~, ~\>:  V\times\BR\Phi \to \BR$ is the natural pairing. Let $V^+$ be the set of dominant vectors $v \in V$, that is, $\<v,\a_i \> \ge 0$ for $i \in \BS$.

The set $B(\bG)$ is equipped with a natural partial order: $[b] \le [b']$ if and only if $\k([b])=\k([b'])$ and $\nu_b \le \nu_{b'}$. Here $\le$ is the dominance order on the set of dominant elements in $X_*(T)_{\BQ}$, that is, $\nu \le \nu'$ if $\nu'-\nu$ is a non-negative rational linear combination of positive relative coroots. It is proved in \cite[Theorem 7.4]{Chai} that the poset $B(\bG)$ is ranked. For any $[b] \le [b']$ in $B(\bG)$, we denote by $\leng ([b], [b'])$ the length of any maximal chain between $[b]$ and $[b']$. 

Let $\mu$ be a dominant coweight. Let $\mu^\diamond$ be the average of the $\s$-orbit of $\mu$. The set of \emph{neutrally acceptable} $\s$-conjugacy classes is defined by 
$$
B(\bG, \mu)=\{[b] \in B(\bG); \k([b])=\k(\mu), \nu_b \le \mu^\diamond\}. 
$$

For any $i\in\BS$, let $\omega_i \in \BR\Phi$ be the corresponding fundamental weight. For any $\s$-orbit $\CO$ of $\BS$, let $\omega_\CO=\sum_{i \in \CO} \omega_i$. The following length formula is due to Chai (see \cite[Theorem 7.4]{Chai} and \cite[Theorem 3.4]{Vi20}). 

\smallskip

(a) {\it For $[b]\in B(\bG,\mu)$, $\leng([b],[t^{\mu}]) = \sum_{\CO\in \BS/\<\s\>} \lceil \<\mu-\nu_b,\omega_\CO\>\rceil.$}

\subsection{Hodge--Newton indecomposable/irreducible set}\label{sec:HN} For any $\s$-stable subset $J$ of $\BS$, we denote by $\bM_J$ the standard Levi subgroup of $\bG_{\breve F}$ associated with $J$. Let $W_J \subseteq W$ be the parabolic subgroup generated by the simple reflections in $J$.  Then $W_J$ is the Weyl group of $\bM_J$. 
Let $b \in \breve G$. We say that $(\mu, b)$ is {\it Hodge--Newton decomposable} with respect to $\bM_J$ if $I(\nu_b) \subseteq J$ and $\mu^{\diamond} - \nu_{b} \in \sum_{j \in J} \BR_{\ge0} \a_j^\vee$. If $(\mu, b)$ is not Hodge--Newton decomposable with respect to any proper $\s$-stable standard Levi subgroup of $\bG_{\breve F}$, then we say that $[b]$ is {\it Hodge--Newton indecomposable}. Set
$$B(\bG, \mu)_{\indec}=\{[b] \in B(\bG, \mu) ;  [b] \text{ is  Hodge--Newton indecomposable}\}.$$

We say that $(\mu, b)$ is {\it Hodge--Newton $J$-irreducible} if $\mu^{\diamond} - \nu_{b} \in \sum_{j \in J} \BR_{>0} \a_j^\vee$. Set $$B(\bG, \mu)_{J\text{-}\irr}=\{[b] \in B(\bG, \mu) ;  [b] \text{ is  Hodge--Newton $J$-irreducible}\}.$$
We simply write $B(\bG, \mu)_{\irr} = B(\bG, \mu)_{\BS\text{-}\irr}$.

We say that $\mu$ is {\it essentially non-central} with respect to $\bM_J$ if it is non-central on every $\s$-orbit of connected components of $J$. It is easy to see that $B(\bG, \mu)_{J\text{-}\irr}\ne\emptyset$ if and only if $\mu$ is essentially non-central with respect to $\bM_J$. We may simply say that $\mu$ is essentially non-central if it is essentially non-central with respect to $\bG$. If $\mu$ is essentially non-central, then $B(\bG, \mu)_{\irr} = B(\bG, \mu)_{\indec}$.

Let $\mathfrak{M}_{\mu}$ be the set of $\s$-stable subsets $J \subseteq \BS$ such that $\mu$ is essentially non-central in $J$. Then we have
$$B(\bG, \mu)=\sqcup_{J \in \mathfrak{M}_{\mu}} B(\bG, \mu)_{J\text{-}\irr}.$$ 

Let $J = \s(J) \subseteq \BS$. For $b \in \bM_J(\breve F)$, we denote by $[b]_{\bM_J}$ the $\s$-conjugacy class of $b$ in $\bM_J(\breve F)$, and denote by $\nu_b^{\bM_J}$ its $\bM_J$-dominant Newton point. We have the following result.

\begin{lemma} \label{Newton}
Let $\mu$ be a dominant coweight and $J$ be a $\s$-stable subset of $\BS$. Then
\begin{enumerate}
\item the map $\phi_J:B(\bM_J, \mu) \to B(\bG, \mu), ~ [b]_{\bM_J} \mapsto [b]$ is injective;
\item the image of $\phi_J$ consists of $[b] \in B(\bG, \mu)$ with $\mu^\diamond - \nu_b \in \sum_{i \in J}\BR_{\ge 0} \a_i^\vee$;
\item for $[b]_{\bM_J} \in B(\bM_J, \mu)$, $\leng_{\bG}([b], [t^\mu]) = \leng_{\bM_J}([b]_{\bM_J}, [t^\mu]_{\bM_J})$.
\end{enumerate}
\end{lemma}
\begin{proof}
Let $[b]_{\bM_J} \in B(\bM_J, \mu)$. Then $\mu^\diamond - \nu_b^{\bM_J} \in \sum_{i\in J}\BR_{\ge 0}\a_i^{\vee}$, which implies that $\nu_b^{\bM_J}$ is dominant with respect to $\bG$, and hence $\nu_b^{\bM_J} = \nu_b$. Now the Newton point and the Kottwitz point of $[b]_{\bM_J}$ are determined by $[b]$ and $\mu$ respectively. Hence $\phi_J$ is injective. 

Part (2) follows from \cite[\S7.1]{Chai} and \cite[Lemma 3.5]{HN2}. Part  (3) follows from part (2) and  Chai's length formula \S\ref{sec:2.2}(a).
\end{proof}

As a consequence, we have the following.

\begin{corollary}\label{cor:irr}
Let $J\in\mathfrak{M}_{\mu}$. Then the map $\phi_J$ in Lemma~\ref{Newton} induces a bijection $B(\bM_J,\mu)_{\irr} \cong B(\bG,\mu)_{J\text{-}\irr}$. 
\end{corollary}

\subsection{Affine Deligne--Lusztig varieties}\label{sec:adlv}
Let $Fl=\breve G/\breve \CI$ be an {\it affine flag variety}. For any $b \in \breve G$ and $w \in \tW$, we define the corresponding {\it affine Deligne--Lusztig variety} in the affine flag variety $$X_w(b)=\{g \breve \CI \in \breve G/\breve \CI; g \i b \s(g) \in \breve \CI \dot w \breve \CI\} \subset Fl.$$

Let $\kk$ be the residue field of $\breve F$. In the equal characteristic setting, the affine Deligne--Lusztig variety $X_w(b)$ is the set of $\kk$-valued points of a locally closed subscheme of the affine flag variety, equipped with the reduced scheme structure. In the mixed characteristic setting, we consider $X_w(b)$ as the $\kk$-valued points of a perfect scheme in the sense of Zhu \cite{Zhu} and Bhatt--Scholze \cite{BS}, a locally closed perfect subscheme of the $p$-adic partial flag variety. 

We denote by $\Sigma^{\mathrm{top}}(X_w(b))$ the set of top-dimensional irreducible components of $X_w(b)$. Let $\bJ_b$ be the $\s$-centralizer of $b$ and let $\bJ_b(F)=\{g \in \breve G; g b \s(g) \i=b\}$ be the group of $F$-points of $\bJ_b$. The left action of $\bJ_b(F)$ on $X_w(b)$ induces an action of $\bJ_b(F)$ on $\Sigma^{\mathrm{top}}(X_w(b))$. We denote by $\bJ_b(F) \backslash \Sigma^{\mathrm{top}}(X_w(b))$ the set of $\bJ_b(F)$-orbits on $\Sigma^{\mathrm{top}}(X_w(b))$. 

Note that if $b$ and $b'$ are $\s$-conjugate in $\breve G$, then $X_w(b)$ and $X_w(b')$ are isomorphic. Thus the affine Deligne--Lusztig variety $X_w(b)$ (up to isomorphism)  depends only on the element $w$ in the Iwahori--Weyl group $\tW$ and the $\s$-conjugacy class $[b]$ in $B(\bG)$. We set $$B(\bG)_w=\{[b] \in B(\bG); X_w(b) \neq \emptyset\}.$$

There is a unique maximal $\s$-conjugacy class in $B(\bG)_w$, which we denote by $[b_{w}]$. By~\cite[Lemma 3.2]{MV}, $\dim X_{w}(b_{w}) = \ell(w) - \<{\nu}_{b_{w}},2\rho\>$. Here $\rho$ is the half sum of the positive roots in $\Phi$.

\subsection{Cordial elements}\label{sec:cordial}
For $b \in \breve G$, the \emph{defect} of $b$ is defined to be $$\de(b)=\rank_F \bG-\rank_F \bJ_b,$$ where $\rank_F$ denotes the $F$-rank of a reductive group over $F$.

By \cite[Theorem 3.4]{Vi20}, we have the following length formula: 

\smallskip

(a) For $[b] \in B(\bG, \mu)$, $\leng([b],[t^{\mu}]) = \<\mu - \nu_b,\rho\> + \frac{1}{2}\de(b).$

\smallskip 

Let ${}^\BS \tW$ be the set of minimal length representatives for the cosets in $W \backslash \tW$. Note that any element $w\in \tW$ can be written in a unique way as $w = xt^{\mu}y$ with $\mu$ dominant, $x,y\in W$ such that $t^{\mu}y\in {}^{\BS}\tW$. We have $\ell(w)=\ell(x)+\<\mu,2\rho\>-\ell(y)$. In this case, we set $\eta_{\s}(w) = \s^{-1}(y)x$. The {\it virtual dimension} is defined to be 
$$d_{w}(b) = \tfrac{1}{2}(\ell(w)+\ell(\eta_{\s}(w) ) - \de(b) - \<\nu_b,2\rho\>).$$ 
For $w = t^{\mu}y \in {}^{\BS}\tW$, it is easy to see that $d_w(b) = \<\mu - \nu_b, \rho\> - \frac{1}{2}\de(b)$.

By \cite[Theorem 10.3]{He14} and \cite[Theorem 2.30]{He-CDM}, we have

\smallskip

(b) For $w\in \tW$ and $b\in \breve G$, $\dim X_{w}(b) \leq d_{w}(b)$.

\smallskip

\noindent In the special case $[b] = [b_w]$, (b) implies that
$$\ell(w) - \ell(\eta_{\s}(w)) \leq \<\nu_{b_{w}}, 2\rho\> - \de(b_{w}).$$

Cordial elements were introduced by Mili\'cevi\'c and Viehmann in \cite{MV}. By definition, an element $w \in \tW$ is {\it cordial} if $\dim X_{w}(b_{w}) = d_{w}(b_{w})$. This condition is equivalent to the condition that $\ell(w) - \ell(\eta_{\s}(w)) = \<\nu_{b_{w}}, 2\rho\>-\de(b_{w})$. The following nice properties of the cordial elements are established in \cite{MV}. 

\begin{theorem}\label{saturated}
Let $w \in \tW$ be a cordial element. Then 
\begin{enumerate}
\item  $B(\bG)_w$ is saturated, that is, if $[b_1]\le[b_2]\le[b_3]$ in $B(\bG)$ and $[b_1],[b_3]\in B(\bG)_w$, then $[b_2]\in B(\bG)_w$; 
\item for each $[b]\in B(\bG)_w$, $X_w(b)$ is equi-dimensional of dimension equal to $d_w(b)$.
\end{enumerate}
\end{theorem}

\subsection{Minimal length elements} We consider the $\s$-conjugation action on $\tW$ defined by $w \cdot_\s w'=w w' \s(w) \i$. Let $B(\tW, \s)$ be the set of $\s$-conjugacy classes of $\tW$. For any $\s$-conjugacy class $\CO$ of $\tW$, we let $\CO_{\min}$ be the set of minimal length elements in $\CO$, and we write $\ell(\CO)=\ell(w)$ for any $w \in \CO_{\min}$.  

For $w, w' \in \tW$ and $i \in \tilde \BS$, we write $w \xrightarrow{s_i}_\s w'$ if $w'=s_i w \s(s_i)$ and $\ell(w') \le \ell(w)$. We write $w \to_\s w'$ if there is a sequence $w=w_0, w_1, \ldots, w_n=w'$ of elements in $\tW$ such that for any $k$, $w_{k-1} \xrightarrow{s_i}_\s w_k$ for some $s_i \in \tilde \BS$. We write $w \approx_\s w'$ if $w \to_\s w'$ and $w' \to_\s w$. 

We call $w, w' \in \tW$ {\it elementarily strongly $\s$-conjugate} if $\ell(w)=\ell(w')$ and there exists $x \in \tW$ such that $w'=x w \s(x) \i$ and $\ell(x w)=\ell(x)+\ell(w)$ or $\ell(w \s(x) \i)=\ell(x)+\ell(w)$. We call $w, w'$ {\it strongly $\s$-conjugate} if there is a sequence $w=w_0, w_1, \ldots, w_n=w'$ such that for each $i$, $w_{i-1}$ is elementarily strongly $\s$-conjugate to $w_i$. We write $w \sim_\s w'$ if $w$ and $w'$ are strongly $\s$-conjugate. 

The following result is proved in \cite[Theorem 2.10]{HN}.

\begin{theorem}\label{min} Let $\CO$ be a $\s$-conjugacy class in $\tW$. Then the following hold:
\begin{enumerate}
\item For each element $w \in \CO$, there exists $w' \in \CO_{\min}$ such that $w \rightarrow_\s w'$.
\item Let $w, w' \in \CO_{\min}$. Then $w \sim_\s w'$.
\end{enumerate}
\end{theorem}

\subsection{Decompositions of affine Deligne--Lusztig varieties}
We recall the Deligne--Lusztig reduction method on  affine Deligne--Lusztig varieties.

\begin{proposition}\label{DLReduction1}\cite[Proposition 3.3.1]{HZZ}
		Let $w\in \tW$, $i\in \tS$, and $b \in \breve G$. If $\operatorname{char}(F) >0$, then the following two statements hold.
		\begin{enumerate}
			\item	If $\ell(s_iw\s(s_i)) =\ell(w)$, then there exists a $\bJ_b(F)$-equivariant morphism $X_w(b) \to X_{s_iw\s(s_i)}(b)$ which is a universal homeomorphism.
			\item
			If $\ell(s_iw\s(s_i)) =\ell(w)-2$, then $X_w(b)=X_1 \sqcup X_2$, where $X_1$ is a $\bJ_b(F)$-stable open subscheme $X$ of $X_w(b)$ and $X_2$ is its closed complement satisfying the following conditions:
			\begin{itemize}
				\item $X_1$ is $\bJ_b(F)$-equivariant universally homeomorphic to a Zariski-locally trivial $\BG_m$-bundle over $X_{s_i w}(b)$;
		        \item $X_2$ is $\bJ_b(F)$-equivariant universally homeomorphic to a Zariski-locally trivial $\BA^1$-bundle over $X_{s_i w \s(s_i)}(b)$.
			\end{itemize} 
		\end{enumerate}
		If $\operatorname{char}(F) =0$, then the above two statements still hold, but with $\BA^1$ and $\BG_m$ replaced by the perfections $\BA^{1, \mathrm{perf}}$ and $\BG^{\mathrm{perf}}_m$, respectively. 
\end{proposition}

For any $a_1, a_2 \in \BN$, we say that a scheme $X$ is an {\it iterated fibration} of type $(a_1, a_2)$ over a scheme $Y$ if there exist morphisms $$X=Y_0 \to Y_1 \to \cdots \to Y_{a_1+a_2}=Y$$ such that for any $i$ with $0 \le i<a_1+a_2$, $Y_i$ is a Zariski-locally trivial $\BA^{1, (\mathrm{perf})}$-bundle or $\BG^{(\mathrm{perf})}_m$-bundle over $Y_{i+1}$, and there are exactly $a_1$ locally trivial $\BG^{(\mathrm{perf})}_m$-bundles in the sequence. In this case, there are exactly $a_2$ locally trivial $\BA^{1, (\mathrm{perf})}$-bundles in the sequence.

Now combining Theorem~\ref{min} with Proposition~\ref{DLReduction1}, we have the following.

\begin{proposition}
Let $w \in \tW$ and $b \in \breve G$. Then $X_w(b)$ admits a decomposition into finitely many locally closed, $\bJ_b(F)$-stable subschemes $X_i$, and each $X_i$ is $\bJ_b(F)$-equivariant universally homeomorphic to an iterated fibration of type $(a_1, a_2)$ over $X_{w'}(b)$ for some $w'$ of minimal length in its $\s$-conjugacy class in $\tW$ and some $a_1, a_2 \in \BN$.
\end{proposition}

\subsection{Classical Deligne--Lusztig varieties}\label{sec:classicalDL} Let $\BF_q$ be a finite field and let $\bar \BF_q$  an algebraic closure of $\BF_q$. Let $\bH$ be a connected reductive group over $\BF_q$ and $H=\bH(\bar \BF_q)$. Let $\s_{\bH}$ be the Frobenius morphism on $H$. Let $B$ be a $\s_{\bH}$-stable Borel subgroup of $H$. Let $W_{H}$ be the Weyl group of $H$ and let $\BS_H$ be the set of simple reflections. Then $\s_{\bH}$ induces a group automorphism on $W_H$ preserving $\BS_{H}$. 
(Classical) Deligne--Lusztig varieties were introduced by Deligne and Lusztig in \cite{DL}. They are defined as follows. For $x \in W_H$, we set $$X^{\bH}_x=\{h B \in H/B; h \i \s(h) \in B \dot x B\}.$$
If $x$ is a $\s_{\bH}$-Coxeter element of $W_H$, then we say that $X^{\bH}_x$ is a {\it classical Deligne--Lusztig variety of Coxeter type} for $\bH$. It is well known that  classical Deligne--Lusztig varieties of Coxeter type are irreducible. 

It is proved in \cite[Theorem 4.8]{He14} that if $w \in \tW$ is a minimal length element in its $\s$-conjugacy class, then $X_w(b) \neq \emptyset$ if and only if $b$ and $\dot w$ are in the same $\s$-conjugacy class of $\breve G$. In this case, \[\tag{a} X_w(b) \cong \bJ_b(F) \times^P X.\] Here $\cong$ means a $\bJ_b(F)$-equivariant universal homeomorphism, $P$ is a parahoric subgroup of $\bJ_b(F)$, and $X$ is (the perfection of) a classical Deligne--Lusztig variety for some connected reductive group $\bH$ over $\BF_q$ with $\bH(\BF_q)$ isomorphic to the reductive quotient of $P$. The group $P$ acts on $\bJ_b(F) \times X$ by $p \cdot (g, z)=(g p \i, p \cdot z)$,  and $\bJ_b(F) \times^P X$ is the quotient space. 

\subsection{Very special parahoric subgroups}\label{sec:special} We follow \cite[\S 2]{HZZ} for the definition of very special parahoric subgroups. 

Let $\bG_1$ be a (not necessarily quasi-split) connected reductive group over $F$ and let $\s_1$ be its Frobenius morphism. Let $\tW_1$ be the Iwahori--Weyl group $\bG_1$. We still denote by $\s_1$ the action on $\tW_1$ induced from the Frobenius morphism on $\bG_1$. Let $\tilde \BS_1$ be the set of simple reflections in $\tW_1$. 

A parahoric subgroup $P$ of $\bG_1(F)$ is called {\it very special} if it is of maximal volume among all the parahoric subgroups of $\bG_1(F)$. A $\s_1$-stable subset $\breve K \subseteq \tilde \BS_1$ is called {\it very special} with respect to $\s_1$ if the parabolic subgroup $W_{\breve K}$ generated by $\breve K$ is finite and the longest element of $W_{\breve K}$ is of maximal length among all such $\s_1$-stable subsets of $\tilde \BS_1$. By \cite[Proposition 2.2.5]{HZZ}, 

\smallskip

(a) {\it A $\s_1$-stable subset $\breve K$ of $\tilde \BS_1$ is very special if and only if $\breve P_{\breve K} \cap \bG_1(F)$ is a very special parahoric parahoric subgroup of $\bG_1(F)$, where $\breve P_{\breve K}$ is the parahoric subgroup of $\bG_1(\breve F)$ corresponding to $\breve K$.} 

\subsection{Statement of the main result}
Let $w\in\tW$. We say that $w$ {\it has finite $\s$-Coxeter part} if $\eta_{\s}(w)$ is a partial $\s$-Coxeter element. In this case, we set $J(w) = \supp_{\s}(\eta_{\s}(w))$. For any $\s$-stable subset $J$ of $J(w)$, denote by $J_{\mu}'$ (resp. $J_{\mu}''$)  the union of all $\s$-orbits of connected components of $J$ in which $\mu$ is non-central (resp. central). Then $\mu$ is essentially non-central in $J_{\mu}'$. By \S\ref{sec:HN}, $B(\bG, \mu)_{J'_\mu\text{-}\irr} \neq \emptyset$, and we have a natural bijection $B(\bM_{J_{\mu}'}, \mu)_{\irr} \cong B(\bG, \mu)_{J'_\mu\text{-}\irr}$. 

Let $J_0(w)$ be the subset of $\BS$ with $\mu - \nu_{b_w}\in \sum_{i\in J_0(w)} \BR_{>0} \a_i^{\vee}$. Then $J_0(w) = J_0(w)'_\mu$. By definition, $[\dot w] \le [b_w]$. Thus $\nu_{\dot w} \le \nu_{b_w}$ and $J_0(w) \subseteq J(w)$.  

For any $\s$-stable $J_1,J_2\subseteq \BS$, denote by $[J_1,J_2] $ the set of $\s$-stable subsets $J\subseteq \BS$ such that $J_1\subseteq J \subseteq J_2$ and denote by $[J_1,J_2]_{\mu} $ the set of $J\in [J_1,J_2]$ such that $\mu$ is essentially non-central in $J$ (or equivalently, $J=J_{\mu}'$).

Now we state the main result of this paper. 

\begin{theorem}\label{main}
Let $w \in Wt^{\mu}W $ such that $\eta_\s(w)$ is a partial $\s$-Coxeter element. Then 
\begin{enumerate}
    \item $w$ is a cordial element;
    \item $B(\bG)_{w}=\sqcup_{J\in [J_0(w),J(w)]_{\mu}} B(\bG, \mu)_{J\text{-}\irr}$;
    \item for any $[b] \in B(\bG)_w$, we have $$X_{w}(b) \cong \bJ_b(F) \times^{P} Y,$$ where $P$ is a parahoric subgroup of $\bJ_b(F)$, and $Y$ is an iterated fibration over (the perfection of) a classical Deligne--Lusztig variety of Coxeter type for some connected reductive group $\bH$ over $\BF_q$ with $\bH(\BF_q)$ isomorphic to the reductive quotient of $P$. 
\end{enumerate}
\end{theorem}

\begin{remark}\leavevmode
\begin{enumerate}
\item In particular, $\bJ_b(F)$ acts transitively on the set of irreducible components of $X_{w}(b)$. 
\item We also have an explicit description of the parahoric subgroup $P$ of $\bJ_b(F)$ in Theorem~\ref{main}(3) (see \S\ref{sec:final}). In particular, if $w \in {}^{\BS} \tW$ and $\eta_\s(w)$ is a (full) $\s$-Coxeter element, then the parahoric subgroup $P$ is very special. 
\end{enumerate}
\end{remark}

Parts (1) and (2) will be proved in \S\ref{sec:cor}. Part~(3) is the most difficult part of this paper and will be proved in \S\S\ref{sec:red} and~\ref{sec:final}. The proof is based on a deep analysis of the reduction tree of $w$, which will be introduced in \S\ref{sec:tr}.

\section{Class polynomials and reduction trees}\label{sec:tr}

In this section, we recall the class polynomials of  Hecke algebras and the connection with  affine Deligne--Lusztig varieties discovered in \cite{He14}. We then introduce the reduction tree, which encodes more information than the class polynomials. 

\subsection{Hecke algebras and their cocenters}
Let $\bq$ be an indeterminate. Let $H$ be the Hecke algebra associated with $\tW$, that is, it is the $\BZ[\bq^{\pm 1}]$-algebra generated by $T_w$ for $w \in \tW$ subject to the relations
\begin{itemize}
    \item $T_w T_{w'}=T_{w w'}$ if $\ell(w w')=\ell(w)+\ell(w')$;
    \item $(T_{s_i}+1)(T_{s_i}-\bq)=0$ for $i \in \tilde \BS$. 
\end{itemize} 

The action of $\s$ on $\tW$ induces an action on $H$, which we still denote by $\s$. The {\it $\s$-commutator} $[H, H]_{\s}$ is the $\BZ[\bq^{\pm 1}]$-submodule generated by $h h'-h' \s(h)$ for $h, h' \in H$. The {\it $\s$-cocenter} of $H$ is defined to be $$\bar H_\s=H/[H, H]_\s.$$

By Theorem~\ref{min}(2), for any $\s$-conjugacy class $\CO$ of $\tW$ and $w, w' \in \CO_{\min}$, we have $T_w+[H, H]_\s=T_{w'}+[H, H]_\s$. We write $T_\CO=T_w+[H, H]_\s \in \bar H_\s$ for any $w \in \CO_{\min}$. It is proved in \cite[Theorem 6.7]{HN} that

\begin{theorem}
\label{basis}
$\bar H_\s$ is a free $\BZ[\bq^{\pm 1}]$-module with basis $\{T_\CO\}_{\CO \in B(\tW, \s)}$. 
\end{theorem}

Moreover, by \cite[\S 2.3]{He14} and \cite[\S 2.8.2]{He-CDM}, for any $w \in \tW$ and $\CO \in B(\tW, \s)$, there exists a unique polynomial $F_{w, \CO} \in \BN[q-1]$ such that $$T_w+[H, H]_\s=\sum_{\CO \in B(\tW, \s)} F_{w, \CO} T_\CO \in \bar H_\s.$$
The polynomials $F_{w, \CO}$ are called  {\it class polynomials}\footnote{The polynomials we use here coincide with those in \cite{He-CDM} and differ from the polynomials used in \cite{He14} by a certain monomial. See \cite[footnote on p. 106]{He-CDM}.}.

\subsection{Class polynomials and affine Deligne--Lusztig varieties}

The class polynomials encode a lot of information about  affine Deligne--Lusztig varieties. 

Let $\tW \to B(\bG)$ be the map sending $w \in \tW$ to the $\s$-conjugacy class $[\dot w]$ of $\breve G$. It is known that this map is independent of the choice of the representative $\dot w$ of $w$, and it induces a map $$\Psi: B(\tW, \s) \to B(\bG).$$ By \cite[Theorem 3.7]{He14}, the map $\Psi$ is surjective. 

Let $w \in \tW$ and $[b] \in B(\bG)$. We set $$F_{w, [b]}=\sum_{\CO \in B(\tW, \s); \Psi(\CO)=[b]} \bq^{\ell(\CO)} F_{w, \CO} \in \BN[\bq-1].$$ Here $\ell(\CO)=\ell(x)$ for any $x \in \CO_{\min}$. 

The following ``dimension = degree'' theorem is established in \cite[Theorem 6.1]{He14}. 

\begin{theorem}
Let $w \in \tW$ and $b \in \breve G$. Then $\dim X_w(b)=\deg F_{w, [b]} - \<\nu_b,2\rho\>$. 
\end{theorem}

\begin{remark}
Here, by convention, $\dim \emptyset=\deg 0=-\infty$. 
\end{remark}

Moreover, we have the following ``leading coefficients = irreducible components'' theorem. This is established in \cite[Theorem 2.19]{He-CDM}. See also \cite[Theorem 3.3.9 and Corollary 3.3.11]{HZZ}. 

\begin{theorem}
Let $w \in \tW$ and $b \in \breve G$. Then the cardinality of $\bJ_b(F) \backslash \Sigma^{\mathrm{top}}(X_w(b))$ equals the leading coefficient of $F_{w, [b]}$. 
\end{theorem}

Although not needed in this paper, it is also worth mentioning that in the superbasic case, the class polynomial gives the number of rational points of an affine Deligne--Lusztig variety. This is established in \cite[Proposition 8.3]{He14}. 

\begin{proposition}
Suppose that the residue field of $F$ is $\BF_q$. Assume that $\bG=PGL_n$ split over $F$ and $b \in \bG(F)$ is a superbasic element in $\breve G$. Then $$\sharp X_w(b)^\s=n F_{w, [b]}\vert_{\bq=q}.$$
\end{proposition}

\subsection{An identity on the class polynomials}

We have the following identity on the class polynomials. We first give a proof using  representations of Hecke algebras. Then we  provide a geometric interpretation of this identity. 

\begin{proposition}\label{prop:id}
Let $w \in \tW$. Then $$\bq^{\ell(w)}=\sum_{\CO \in B(\tW, \s)} \bq^{\ell(\CO)} F_{w, \CO}=\sum_{[b] \in B(\bG)} F_{w, [b]}.$$
\end{proposition}

\begin{proof}
We prove the first equality. The second  follows from the definition. 

Let $\pi: H \to \BZ[\bq^{\pm 1}]$ be the homomorphism of $\BZ[\bq^{\pm 1}]$-algebras sending $T_{s_i}$ to $\bq$ for any $i \in \tilde \BS$. As $\pi \circ \s = \pi$, $\pi([H, H]_{\s})=0$ and thus $\pi$ induces a homomorphism of the $\BZ[\bq^{\pm 1}]$-modules $\bar H_\s \to \BZ[\bq^{\pm 1}]$, which we still denote by $\pi$. 

We have $T_w+[H, H]_\s=\sum_{\CO \in B(\tW, \s)} F_{w, \CO} T_\CO$. Applying $\pi$ to both sides, we obtain $\bq^{\ell(w)}=\sum_{\CO \in B(\tW, \s)} \bq^{\ell(\CO)} F_{w, \CO}$. 
\end{proof}

In the rest of this subsection, we assume furthermore that $F=\BF_q((\e))$ and that $\bG$ is split over $F$. We give a geometric interpretation of the above identity. 

For any $n \in \BN$, let $\breve I_n$ be the $n$th congruence subgroup of $\breve I$. Following \cite[\S 2.10]{GHKR06}, we call a subset $X$ of $\breve G$ {\it admissible} if for any $w \in \tW$, there exists $n \in \BN$ such that $X \cap \breve I \dot w \breve I$ is stable under  right multiplication of $\breve I_n$. In this case, the action of $\breve I_n$ on $X \cap \breve I \dot w \breve I$ is free, and $\frac{\sharp \bigl((X \cap \breve I \dot w \breve I)/\breve I_n\bigr)^\s}{\sharp \breve I_n^\s}$ is independent of the choice of such $n$. We set $$\breve \sharp (X \cap \breve I \dot w \breve I)=\frac{\sharp \bigl((X \cap \breve I \dot w \breve I)/\breve I_n\bigr)^\s}{\sharp \breve I_n^\s}.$$

An admissible subset $X$ is called {\it bounded} if $X \cap \breve I \dot w \breve I=\emptyset$ for all but finitely many $w \in \tW$. For any bounded admissible subset $X$, we set $$\breve \sharp X=\sum_{w \in \tW} \breve \sharp (X \cap \breve I \dot w \breve I)$$ and call it the {\it normalized cardinality of the rational points} of $X$. 
By \cite[Theorem A.1]{He-KR}, each $\s$-conjugacy class of $\breve G$ is admissible. We have the following geometric interpretation of the class polynomials. 

\begin{proposition}
Let $w \in \tW$ and $[b] \in B(\bG)$. Then $$\breve \sharp ([b] \cap \breve I \dot w \breve I)=F_{w, [b]} \vert_{\bq=q}.$$
\end{proposition}

\begin{proof}
We argue by induction on $\ell(w)$. 

If $w$ is a minimal length element in its $\s$-conjugacy class in $\tW$, then $$[b] \cap \breve I \dot w \breve I=\begin{cases} \breve I \dot w \breve I & \text{if } [b]=[\dot w], \\ \emptyset & \text{otherwise}. \end{cases}$$

On the other hand, by \cite[\S 2.8.2]{He-CDM}, $$F_{w, [b]}=\begin{cases} \bq^{\ell(w)} & \text{if } [b]=[\dot w], \\ 0 & \text{otherwise}. \end{cases}$$ Thus the proposition holds in this case. 

Now we assume that $w$ is not a minimal length element in its $\s$-conjugacy class. By Theorem~\ref{min}(1), there exist $w' \in \tW$ and $i \in \tilde \BS$ such that $w \approx_\s w'$ and $s_i w' \s(s_i)<w'$. Set $w_1=s_i w'$ and $w_2=s_i w' \s(s_i)$. Then $\ell(w_1), \ell(w_2)<\ell(w)$. By the proof of \cite[Theorem 2.16]{He-CDM}, $$\breve \sharp ([b] \cap \breve I \dot w \breve I)=\breve \sharp ([b] \cap \breve I \dot w' \breve I)=(q-1) \breve \sharp ([b] \cap \breve I \dot w_1 \breve I)+q \breve \sharp ([b] \cap \breve I \dot w_2 \breve I).$$ On the other hand, by \cite[\S 2.8.2]{He-CDM}, $F_{w, [b]}=(\bq-1) F_{w_1, [b]}+\bq F_{w_2, [b]}$. Now the statement for $w$ follows from the inductive hypothesis on $w_1$ and $w_2$. 
\end{proof}

\smallskip

We have the decomposition $\breve I \dot w \breve I=\sqcup_{[b] \in B(\bG)} [b] \cap \breve I \dot w \breve I$. Thus $$\bq^{\ell(w)}=\breve \sharp (\breve I \dot w \breve I)=\sum_{[b] \in B(\bG)} \breve \sharp([b] \cap \breve I \dot w \breve I)=\sum_{[b] \in B(\bG)} F_{w, [b]}.$$

This gives an alternative proof of Proposition~\ref{prop:id}  and a geometric interpretation of Proposition~\ref{prop:id} in the case where the $\s$-action on $\tW$ is trivial via counting (the normalized cardinality of) the rational points of $\breve I \dot w \breve I$. 

\subsection{Reduction trees}\label{sec:tree}
Let $w \in \tW$. We construct the reduction tree for $w$, which encodes the Deligne--Lusztig reduction for the affine Deligne--Lusztig varieties associated with $w$ (and with all $b \in \breve G$). 

The vertices of the graphs are the elements of $\tW$, and the (oriented) edges are of the form $x \lup y$, where $x, y \in \tW$, and there exists $x' \in \tW$ and $i \in \tilde \BS$ with $x \approx_\s x'$, $s_i x' \s(s_i)<x'$ and $y \in \{s_i x', s_i x' \s(s_i)\}$. Some elements of $\tW$ may occur more than once in a reduction tree. 

The reduction trees are constructed inductively. 

Suppose that $w$ is of minimal length in its $\s$-conjugacy class of $\tW$. Then the reduction tree of $w$ consists of a single vertex $w$ and no edges. 

Suppose that $w$ is not of minimal length in its $\s$-conjugacy class of $\tW$ and that a reduction tree is given for any $z \in \tW$ with $\ell(z)<\ell(w)$. By Theorem~\ref{min}(1), there exist $w' \in \tW$ and $i \in \tilde \BS$ with $w \approx_\s w'$ and $s_i w' \s(s_i)<w'$. The reduction tree of $w$ is the graph containing the given reduction tree for $s_i w'$ and the reduction tree for $s_i w' \s(s_i)$, and the edges $w \lup s_i w'$ and $w \lup s_i w' \s(s_i)$. 

Note that the reduction trees of $w$ are not unique. They depend on the choices of $w'$ and $s_i$ in the construction. We will see in the rest of this section that the reduction trees encode more information than the class polynomials. 

\subsection{Reduction path}
Let $\CT$ be a reduction tree of $w$. An {\it end point} of the tree $\CT$ is a vertex $x$ of $\CT$ such that there is no edge of the form $x \lup x'$ in $\CT$. By Theorem~\ref{min}, each end point is of minimal length in its $\s$-conjugacy class. A {\it reduction path} in $\CT$ is a path $\underline p: w \lup w_1 \lup \cdots \lup w_n$, where $w_n$ is an end point of $\CT$. The {\it length} $\ell(\underline p)$ of the reduction path $\underline p$ is the number of edges in $\underline p$. We also write $\operatorname{end}(\underline p)=w_n$ and $[b]_{\underline p}=\Psi(\operatorname{end}(\underline p)) \in B(\bG)$.

Note that if $x \lup y$, then $\ell(x)-\ell(y) \in \{1, 2\}$. We say that the edge $x \lup y$ is of type I if $\ell(x)-\ell(y)=1$ and  of type II if $\ell(x)-\ell(y)=2$. For any reduction path $\underline p$, we denote by $\ell_I(\underline p)$ the number of type I edges in $\underline p$ and by $\ell_{II}(\underline p)$ the number of type II edges in $\underline p$. Then $\ell(\underline p)=\ell_I(\underline p)+\ell_{II}(\underline p)$. 

The following relation between  class polynomials and reduction trees follows easily from the inductive construction, and we omit the details of its proof. 

\begin{lemma}\label{class-poly}
Let $w \in \tW$ and let $\CT$ be a reduction tree of $w$. Then, for any $\s$-conjugacy class $\CO$ of $\tW$, we have $$F_{w, \CO}=\sum_{\underline p} (\bq-1)^{\ell_I(\underline p)} \bq^{\ell_{II}(\underline p)},$$ where $\underline p$ runs over all the reduction paths in $\CT$ with $\Psi(\operatorname{end}(\underline p)) = \CO$. 
\end{lemma}

Combining Proposition~\ref{DLReduction1} with the construction of the reduction trees, we obtain the following result. 

\begin{proposition}\label{prop:dec}
Let $w \in \tW$ and $\CT$ be a reduction tree of $w$. Then, for any $b \in \breve G$, there exists a decomposition $$X_w(b)=\bigsqcup_{\underline p \mathrm{\; is\; a\; reduction\; path\; of\; } \CT; [b]_{\underline p}=[b]} X_{\underline p},$$ where $X_{\underline p}$ is a locally closed subscheme of $X_w(b)$, and is $\bJ_b(F)$-equivariant universally homeomorphic to an iterated fibration of type $(\ell_I(\underline p), \ell_{II}(\underline p))$ over  $X_{\operatorname{end}(\underline p)}(b)$.
\end{proposition}

\begin{remark}\label{fibration}
Since $\operatorname{end}(\underline p)$ is a minimal length element in its $\s$-conjugacy class, by \S\ref{sec:classicalDL}(a) we have $X_{\operatorname{end}(\underline p)}(b) \cong \bJ_b(F) \times^P X$, where $P$ is a parahoric subgroup of $\bJ_b(F)$ and  $X$ is (the perfection of) an irreducible component of a classical Deligne--Lusztig variety. Thus each irreducible component $Y$ of $X_{\underline p}$ is universally homeomorphic to an iterated fibration of type $(\ell_I(\underline p), \ell_{II}(\underline p))$ over $X$. We have a natural action of $\bJ_b(F)$ on $X_w(b)$, and $X_{\underline p}$ is stable under this action. In this case, $X_{\underline p} \cong \bJ_b(F) \times^P Y$. 
\end{remark}

\section{Cordiality and the set $B(\bG)_w$}\label{sec:cor}

\subsection{Maximal Hodge--Newton irreducible elements}
Recall that $[b_{t^\mu}]$ is the unique maximal element of $B(\bG, \mu)$. We have the following result on $B(\bG, \mu)_{\indec}$.

\begin{proposition}\label{prop:max-b}
The set $B(\bG, \mu)_{\indec}$ contains a unique maximal element.
\end{proposition}

\begin{remark}
We denote this element by $[b_{\mu, \bG\text{-}\indec}]$. 
\end{remark}

\begin{proof}
Following \cite{Chai}, for any subset $E \subseteq (V^+)^\s$, we set $$C_{\ge E} = \{v \in (V^+)^\s; v \ge v', \text{ for any } v' \in E\}.$$ 

For any $i \in \BS$, denote by $\CO_i$ the $\s$-orbit of $i$. Set $\o_{\CO_i} = \sum_{j \in \CO_i} \o_j$. Define $e_i \in \BQ \o_i^\vee$ with $$\<e_i,\o_i  \> = \frac{1}{\sharp \CO_i} \max\{0, \<\mu,\o_{\CO_i}\> - 1\}.$$ Set $E = \{e_i; i \in \BS\}$. By \cite[Theorem 6.5]{Chai}, there exists a unique $\s$-conjugacy class $[b_{\mu, \bG\text{-}\indec}] \in B(\bG, \mu)$ such that $\nu_{b_{\mu, \bG\text{-}\indec}} = \min C_{\ge E}$ and $\< \mu - \nu_{b_{\mu, \bG\text{-}\indec}},\o_{\CO_k}\> = 1$ for any $k\in \BS - I(\nu)$. Set $\nu=\nu_{b_{\mu, \bG\text{-}\indec}}$.

Suppose that $(\mu, [b_{\mu, \bG\text{-}\indec}])$ is Hodge--Newton decomposable with respect to some standard Levi subgroup $\bM_J$ with $J = \s(J) \subsetneqq \BS$. Let $j \in \BS - J$. By definition, $\< \mu-\nu,\o_{\CO_j}\> = 0$. On the other hand, as $I(\nu) \subseteq J$ we have $\<\nu,\a_j\> \neq 0$. Thus $\< \nu,\o_{\CO_j}\> = 1$, which is a contradiction. Therefore $[b_{\mu, \bG\text{-}\indec}] \in B(\bG, \mu)_{\indec}$.

On the other hand, let $[b'] \in B(\bG, \mu)_{\indec}$. Set $\nu'= \nu_{b'}$. For any $i \in \BS - I(v')$, we denote by $pr_{(i)}: V=\BR\o^{\vee}_i\oplus \sum_{j\ne i}\BR\a_j^{\vee} \to \BR \o_i^\vee$  the natural projection. Set $e'_i= pr_{(i)}(v') \in \BR \o_i^\vee$. Let $E'=\{e'_i; i \in \BS - I(v')\}$. Again by \cite[Theorem 6.5]{Chai}, we have $v' = \min C_{\ge E'}$. By \S\ref{sec:2.2}(a), we have $\<\mu - \nu',\o_{\CO_i}\> \in \BZ_{\ge 1}$ for $i \in \BS - I(v')$. This means that $e_i' \le e_i$ for $i \in \BS - I(v')$. So $v \in C_{\ge E'}$ and $v  \ge \min C_{\ge E'} = v'$. Hence $[b_{\mu, \bG\text{-}\indec}]$ is the unique maximal element of $B(\bG, \mu)_{\indec}$.
\end{proof}

Combining the description of $[b_{\mu, \bG\text{-}\indec}]$ in the proof above with \S\ref{sec:2.2}(a), we have the following result. 

\begin{corollary}\label{cor:max-b}
Suppose that $\mu$ is essentially non-central. Then $${\leng}([b_{\mu, \bG\text{-}\indec}], [t^{\mu}])=\sharp(\BS/\<\s\>).$$
\end{corollary}

\subsection{Proof of Theorem~\ref{main}(1) and (2) for $w=t^{\mu}c$}\label{tmuc_case}
Now we prove Theorem~\ref{main}(1) and (2) in the case $w = t^{\mu}c\in{}^{\BS}\tW$, where $c$ is a partial $\s$-Coxeter element. In this case, it is easy to see that $\mu$ is essentially non-central in $\supp_{\s}(c)$. Set $J = \supp_{\s}(c)$. We have $J(w) = J(w)_{\mu}' = J_0(w) = J$. We need to show that $w$ is cordial and $B(\bG)_w = B(\bG,\mu)_{J\text{-}\irr}$.

By \S\ref{sec:HN}, we have a natural bijection $B(\bM_J,\mu)_{\irr} \cong B(\bG,\mu)_{J\text{-}\irr}$. By \cite[Theorem 3.3.1]{GHN}, we have $B(\bG)_w = B(\bM_J)_w$. Note that $\mu^{\diamond}-\nu_{b_w}\in \sum_{j\in J}\BR\a_j^{\vee}$. Hence $\<\mu-\nu_{b_w},\rho\> = \<\mu-\nu_{b_w},\rho_{J}\>$, where $\rho_J$ is the half sum of positive roots of $\bM_J$. By definition, $w$ is cordial in $\bG$ if and only if it is cordial in $\bM_J$. Hence we may assume that $J = \BS$ and that $c$ is a $\s$-Coxeter element. 

We first show that

\smallskip

(a) $B(\bG)_w \subseteq B(\bG, \mu)_{\indec}$. 

\smallskip

\noindent Suppose that $X_w(b) \neq \emptyset$ and $(\mu, b)$ is Hodge--Newton decomposable with respect to some proper standard Levi subgroup $\bM$. By \cite[Theorem 1.11]{GHN-full}, there is some $u\in W$ such that $u^{-1}w\s(u)$ lies in $\tW_{\bM}$, which contradicts  the fact that $c$ is $\s-$Coxeter. Thus (a) is proved.


Next we show that 

\smallskip

(b) $[b_w]=[b_{\mu, \bG\text{-}\indec}]$ and $w$ is a cordial element. 

\smallskip

\noindent By \S\ref{sec:cordial}(b), we have 
$$\<\mu, 2 \rho\>-2 \ell(c)=\ell(w)-\ell(\eta_\s(w)) \le \<\nu_{b_w}, 2\rho\>-\de(b_w).$$ Combined with \S\ref{sec:cordial} (a), we get $\leng([b_w],[t^{\mu}]) \le \ell(c)=\sharp(\BS/\<\s\>)$. However, by Corollary~\ref{cor:max-b}, we have $\leng ([b_w],[b_{t^{\mu}}]) \ge \leng ([b_{\mu, \bG\text{-}\indec}],[b_{t^{\mu}}])=\sharp(\BS/\< \s\>)$. Thus we must have $[b_w]=[b_{\mu, \bG\text{-}\indec}]$ and $\dim X_w(b_w)=d_w(b_w)$. Thus (b) is proved. 

Finally, we show that 

\smallskip

(c) $B(\bG)_w=B(\bG, \mu)_{\indec} = B(\bG, \mu)_{\irr}$. 

\smallskip

\noindent Let $[b_{\min}]$ be the unique basic $\s$-conjugacy class in $B(\bG)$ with $\k(b_{\min})=\k(\mu)$. Then $[b_{\min}]$ is the unique minimal element in $B(\bG, \mu)_{\indec}$. Note that $c$ is a $\s$-Coxeter element of $W$. Thus $\nu_{\dot w}$ is central and $[\dot w]=[b_{\min}]$. In particular, $[b_{\min}]\in B(\bG)_w$. By Theorem~\ref{saturated}, $B(\bG)_w$ is saturated and hence must be equal to $B(\bG, \mu)_{\indec}$. 

This completes the proof of the $t^{\mu}c$ case.

\subsection{Partial conjugation}\label{par_conj}
To handle the general case, we use the partial conjugation method introduced in \cite{He07}. 

By partial conjugation, we mean conjugating by elements in the finite Weyl group $W$. For any $x\in  {}^{\BS}\tilde{W},$ set 
$$ I(x) = \max\{ J\subseteq \BS ; \Ad(x)\s(J) = J \}.$$
This is well-defined. Indeed, if $\Ad(x)\s(J_i) = J_i$ for $i = 1,2$, then $\Ad(x)\s(J_1\cup J_2) = J_1\cup J_2$. Let $W_{I(x)}$ be the subgroup of $W$ generated by the simple reflections in $I(x)$. Then $\Ad(x) \circ \s$ gives a length-preserving group automorphism on $W_{I(x)}$. By \cite[Proposition 2.4]{He07}, we have
$$\tW = \bigsqcup_{x\in  {}^{\BS}\tilde{W}}W\cdot_{\s} (W_{I(x)}x) = \bigsqcup_{x\in  {}^{\BS}\tilde{W}}W\cdot_{\s} (x W_{\s(I(x))}).$$ Moreover, by \cite[Proposition 3.4]{He07}, we have the following:

\smallskip

(a) {\it For any $w \in \tW$, there exists $x \in {}^{\BS} \tW$ and $u \in W_{I(x)}$ such that $w \to_\s ux$ and all the simple reflections involved in the conjugations are in $\BS$.}

\smallskip

By \cite[Proposition 4.9]{He14}, we have the following:

\smallskip

(b) {\it Let $x \in {}^{\BS} \tW$ and $u\in W_{I(x)}$. Then $B(G)_{ux}$ = $B(G)_{x}$ and $\dim X_{ux}(b) = \dim X_{x}(b) + \ell(u)$ for any $[b]\in B(\bG)_x$.}

\smallskip

Similar to \S\ref{sec:tree}, we may consider partial reductions. By partial reduction, we mean reduction $w \lup s_i w$ or $w \lup s_i w\s(s_i)$ with $i \in \BS$. We shows that partial reduction preserves elements with finite Coxeter parts.

\begin{lemma}\label{par_cox}
Let $w \in \tW$ with $\eta_\s(w)$ a partial $\s$-Coxeter element of $W$. Let $i \in \BS$ with $s_i w<w$. Then 
\begin{enumerate}
\item $\eta_{\s}(s_i w \s(s_i))$ is a partial $\s$-Coxeter element of $W$ and $\supp_\s(\eta_{\s}(s_i w \s(s_i)))=\supp_\s(\eta_{\s}(w))$. 
\item  If, moreover, $s_i w \s(s_i)<w$, then $\eta_\s(s_i w)$ is a partial $\s$-Coxeter element of $W$ and $\supp_\s(\eta_{\s}(s_i w ))=\supp_\s(\eta_{\s}(w))-\{\s^l(i'); l \in \BZ\}$ for some $i'\in \supp_{\s}(\eta_{\s}(w))$.
\end{enumerate}
\end{lemma}

\begin{proof} We prove part (1). The proof of part (2) is similar, and we skip the details. 

Write $w=xt^{\mu}y$ with $t^{\mu}y\in{}^{\BS}\tW$. Set $c=\eta_{\s}(w)=\s^{-1}(y)x$. If $y\s(s_i)\in {}^{\BS}\tW$, then $\eta_{\s}(s_iw\s(s_i)) = \eta_{\s}(w)$, and the statement is obvious. Now assume that $y\s(s_i) = s_{i'}y$ for some $i'\in I(\mu)$. Then we have $x^{-1}(\a_i) < 0 $ and $\s^{-1}(y)(\a_i) > 0$. Thus $(\s^{-1}(y)x)(-x^{-1}(\a_i)) < 0 $. It follows that $\s^{-1}(y)s_ix = (\s^{-1}(y)x)(x^{-1}s_ix) < \s^{-1}(y)x$. By the cancellation property of Coxeter groups, we conclude that $\s^{-1}(y)s_ix$ is a partial $\s$-Coxeter element and $\ell(\s^{-1}(y)s_ix) = \ell(\s^{-1}(y)x) -1$. Write $c' = \s^{-1}(y)s_ix$. Notice that $\s^{-1}(s_{i'})c = \s^{-1}(s_{i'}y)x  = \s^{-1}(y)s_ix = c'$. It follows that $\s^{-1}(i')$ is not in the $\s$-support of $c'$. Hence $\eta_{\s}(s_iw\s(s_i))=\s^{-1}(y)s_ixs_{i'} = c's_{i'} $ is a partial $\s$-Coxeter element, and $$\supp_\s(\eta_{\s}(s_i w \s(s_i)))=\supp_\s(\eta_{\s}(w)).$$ 
The statement is proved. 
\end{proof}

\subsection{Proof of Theorem~\ref{main}(1) and (2): general case} 
Let $w = xt^{\mu}y\in \tW$ with $t^\mu y \in {}^{\BS} \tW$. We assume that $\eta_{\s}(w)$ a partial $\s$-Coxeter element. We prove Theorem~\ref{main}(1) and (2) by induction on $\ell(x)$. 

The case $\ell(x)=0$ has already been proved in \S\ref{tmuc_case}. Assume that $\ell(x)>0$. Let $i\in \BS$ such that $s_ix<x$. There are three different cases.

\medskip

Case (1): $\ell(s_iw\s(s_i))<\ell(w)$. Write $w_1 = s_iw$ and $w_2 = s_iw\s(s_i)$. By Lemma~\ref{par_cox}, $\eta_{\s}(w_1)$ and $\eta_{\s}(w_2)$ are both partial $\s$-Coxeter elements. The inductive hypothesis applies for $w_1$ and $w_2$. Since $w_1>w_2$, we have $[b_w] = \max([b_{w_1}],[b_{w_2}]) = [b_{w_1}]$. Then $d_w(b_w) = d_{w_1}(b_w) + 1$ and $\dim X_{w_1}(b_w) + 1 = \dim X_w(b_w)$. Thus $w$ is cordial. 

Observe that 
$$B(\bG)_w = B(\bG)_{w_1}\cup B(\bG)_{w_2} =\sqcup_{J\in [J_0(w_1),J(w_1)]_{\mu}\cup [J_0(w_2),J(w_2)]_{\mu}} B(\bG, \mu)_{J\text{-}\irr}. $$
Note that $J(w_2) = J(w)$, $J_0(w_1) =J_0(w)$ and $J(w_1)\subset J(w)$. By \S\ref{sec:cordial}(b), $B(\bG)_w$ is saturated. Hence we must have 
$$[J_0(w_1),J(w_1)]_{\mu}\cup [J_0(w_2),J(w_2)]_{\mu}=[J_0(w),J(w)]_{\mu}.$$
This proves part (2) of Theorem~\ref{main}.

\medskip

Case (2): $y\s(s_i)<y$. Write $w'=s_i w\s(s_i)=s_i xt^{\mu}y\s(s_i)$. The inductive hypothesis applies for $w'$. Note that $w \approx_{\s} w'$, in particular, $B(\bG)_w=B(\bG)_{w'}$. Note also that $J(w)=J(w')$. Hence the statements hold for $w$. 

\medskip

Case (3): $y\s(s_i) = s_{i'} y $ for some $i' \in I(\mu)$ and $\ell(s_i xs_{i'}) = \ell(x)$. Write $w' = s_i w \s(s_i)=s_i x s_{i'} t^{\mu}y$. Then $w \approx_{\s} w'$.  By Lemma~\ref{par_cox}, $\eta_{\s}(w') = \s^{-1}(y)s_i x s_{i'} $ is a partial $\s$-Coxeter element with length equal to $\ell(\s^{-1}(y)x)$. Hence the statements hold for $w$ if and only if they hold for $w'$. We continue the procedure until case (1) or (2) happens. If case (3) happens all the time and the procedure does not end, then $x\in W_{I(t^{\mu}y)}$, and both $x$ and $y$ are partial $\s$-Coxeter elements. Then the statements follow from \S\ref{par_conj}(b).

\section{Analyzing the reduction paths}\label{sec:end}

\subsection{$\s$-Conjugacy classes of $\tW$} 
We first recall the definition of elliptic conjugacy classes. Let $W_1$ be a Coxeter group and let $S_1$ be the index set of simple reflections in $W_1$. Let $\d$ be a length-preserving group automorphism on $W_1$. A $\d$-conjugacy class $C$ of $W_1$ is called {\it elliptic} if it contains no elements in any proper $\d$-stable standard parabolic subgroup of $W_1$. 

Let $x \in \tW$. We regard $x \s$ as an element in $\tW \rtimes \<\s\>$. There exists a positive integer such that $(x \s)^n=t^\l $ for some $\l \in X_*(T)_{\G_0}$. Then we set $\nu_x=\l/n \in V$. It is easy to see that $\nu_x$ is independent of the choice of $n$. Moreover, the unique dominant element in the $W$-orbit of $\nu_x$ equals the (dominant) Newton point $\nu_{\dot x}$ for $\dot x \in \breve G$. 

We follow \cite[\S 1.8.3]{He-CDM}. Let $J \subseteq \BS$. Let $\tW_J=X_*(T)_{\G_0} \rtimes W_J$ be the Iwahori--Weyl group of the standard Levi subgroup $\bM_J$ of $\bG$. Let ${}^{J}\tW$ be the set of minimal length representatives for cosets in $W_J\backslash \tW$. Let $\tilde J \supseteq J$ be the set of simple reflections for the Iwahori--Weyl group $\tW_J$.

We say that $(J, x, \breve K, C)$ is a {\it standard quadruple} if 
\begin{enumerate}

\item $J \subseteq \BS$ with $\s(J)=J$;

\item $x \in {}^J\tW$ such that $\nu_x$ is dominant, $J = I(\nu_x)$, and $\Ad(x) \circ \s$ preserves $\tilde J$;

\item $\breve K \subseteq \tilde J$ with $W_{\breve K}$ finite and $\Ad(x) \d(\breve K)=\breve K$;

\item $C$ is an elliptic $(\Ad(x) \circ \s)$-conjugacy class of $W_{\breve K}$. 
\end{enumerate} 

We say that the standard quadruples $(J, x, \breve K, C)$ and $(J', x', \breve K', C')$ are {\it equivalent} in $\tW$ if $J=J'$, there exists a length-zero element $\t$ of $\tW_J$ with $x'=\t x \s(\t) \i$, and there exists $w \in \tW_J$ with $x' \s(w) (x') \i=w$ and $C'=w \t C (w \t) \i$. 

By \cite[Theorem 1.19]{He-CDM}, we have the following:

\smallskip

(a) {\it The map $(J, x, \breve K, C) \mapsto \tW \cdot_\s C x$ induces a bijection between the equivalence classes of standard quadruples and the set of $\s$-conjugacy classes of $\tW$.}

\smallskip

Let $\CO \in B(\tW, \s)$ and let $(J, x, \breve K, C)$ be a standard quadruple associated with $\CO$. We say that $\CO$ is {\it Coxeter} (resp. {\it elliptic}) associated with $[b] \in B(\bG)$ if $\Psi(\CO)=[b]$, $\breve K \subset \tilde J$ is very special with respect to $\Ad(\dot x) \circ \s$, and $C$ is an $(\Ad(x) \circ \s)$-Coxeter (resp. elliptic) conjugacy class of $W_{\breve K}$. We say that $w \in \tW$ is a $\s$-Coxeter element associated with $[b]$ if it is a minimal length element in a Coxeter $\s$-conjugacy class associated with $[b]$. 

\subsection{Description of reduction trees}\label{sec:5.3}
For any $[b]\in B(\bG,\mu)$, set 
\begin{align*} \ell_I(\mu, [b])&=\sharp(\BS/\<\s\>)- \sharp(I(\nu_b)/\<\s\>), \\
\ell_{II}(\mu, [b])&= \leng ([b],[t^{\mu}]) - \sharp(\BS/\<\s\>).  
\end{align*}
We have the following description of  reduction trees. 

\begin{theorem}\label{thm:main}

Let $c$ be a $\s$-Coxeter element of $W$ such that $t^\mu c \in {}^{\BS} \tW$. Let $\CT$ be a reduction tree of $t^\mu c$. Then, for any reduction path $\underline p$ in $\CT$, we have 
\begin{enumerate}
    \item $\ell_I(\underline p)=\ell_I(\mu, [b]_{\underline p})$ and $\ell_{II}(\underline p)=\ell_{II}(\mu, [b]_{\underline p})$;
    \item $\operatorname{end}(\underline p)$ is a $\s$-Coxeter element associated with $[b]_{\underline p}$.
\end{enumerate}

Moreover, for any $[b] \in B(\bG, \mu)_{\indec}$, there exists a unique reduction path $\underline p$ in $\CT$ with $[b]_{\underline p}=[b]$. 
\end{theorem}

Combining Theorem~\ref{thm:main} with Proposition~\ref{prop:dec} and Remark~\ref{fibration}, we obtain part (3) of Theorem~\ref{main} for $w=t^\mu c$. We will describe the reduction trees of the elements with finite partial $\s$-Coxeter part in \S\ref{sec:final} and deduce Theorem~\ref{main}(3) for such elements. 

In the rest of this section, we shall prove parts (1) and (2) of Theorem~\ref{thm:main}. The ``moreover'' part (i.e., the multiplicity-one result) is 
the most difficult part and will be proved in \S\ref{sec:red}.

\subsection{Estimate $\ell_I$}\label{bound} Let $\operatorname{Aff}(V)$ be the group of affine transformations on $V$. For any $g\in \operatorname{Aff}(V)$, denote $V^{g} = \{v\in V\mid g(v)=v\}$. We have a natural projection map $p: \tW \rtimes \<\s\> \to \operatorname{Aff}(V) \to \GL(V)$. For any $w \in \tW$, define $V_w = \{ v \in V;w\s(v) = v + \nu_w\}$. We have $\dim V_w = \dim V^{p(w \s)}$.  

It is easy to see that for any $g \in \GL(V)$ and any reflection $r \in \GL(V)$, we have $$|\dim V^{r g}-\dim V^g| \le 1.$$ In particular, for any $w \in \tW$ and $i \in \tilde \BS$, we have $$|\dim V_{s_i w}-\dim V_w| \le 1.$$

Now let $w = t^{\mu}c$ be as in Theorem~\ref{thm:main}. Let $\CT$ be a reduction tree of $w$ and let $\underline p$ be a reduction path in $\CT$. Set $e = \operatorname{end}(\underline p)$ and $[b] = \Psi(e) \in B(\bG)$. Let $(J,x,\breve K,C)$ be a standard quadruple associated with the $\s$-conjugacy class of $e$. 

Consider the variation of $\dim V_{-}$ along the reduction path $\underline p$. Note that type II edges do not change $\dim V_{-}$, and  type I edges change $\dim V_{-}$ by at most 1. Therefore $\ell_I(\underline p) \ge \vert\dim V_e-\dim V_w\vert$. Since $p(w)=c$ is a $\s$-Coxeter element, $\dim V^{p(w\s)}=0$. Since $C$ is $\Ad(x)\circ \s$-elliptic in $\breve K$, we have $\dim V_e=\dim V_x-\sharp (\breve K/\<\Ad(x)\circ \s\>)$. Hence \[\tag{a} \ell_I(\underline p) \ge \dim V_x-\sharp (\breve K/\<\Ad(x)\circ \s\>).\]

By \cite[\S 1.9]{Ko06}, $\de (b)=\sharp(\BS/\<\s\>)-\dim V_x$. Note that $\ell_I(\underline p)+2 \ell_{II}(\underline p)=\ell(t^\mu c)-\ell(e)$. Moreover, \[\tag{b} \ell(e)  \ge \<\nu_b, 2 \rho\>+\sharp (\breve K/\<\Ad(x)\circ \s\>),\] with equality holding if and only if $C$ is an $(\Ad(x)\circ \s)$-Coxeter conjugacy class in $\breve K$. We have 
\begin{align*}
\dim X_{\underline p} & =\ell_I(\underline p)+\ell_{II}(\underline p)+\ell(e)-\<\nu_b, 2 \rho\>\\
   & = \tfrac{1}{2}(\ell_I(\underline p) + \<\mu,2\rho\>  - \sharp(\BS/\<\s \>) + \ell(e)) -\<\nu_b, 2 \rho\>\\
   & \ge  \<\mu -\nu_b,\rho\> +  \tfrac{1}{2}(\dim V_x-\sharp(\BS/\<\s\>)) \\
   & =  \<\mu -\nu_b,\rho\>  - \tfrac{1}{2} \de (b)=d_{w}(b).
\end{align*}
By \S\ref{sec:cordial}(b), we have $\dim X_{\underline p} \le \dim X_{w}(b) \le \dim d_w(b)$. Thus the inequalities in (a) and (b) are  equalities, and $\dim X_{\underline p} = \dim X_w(b)$. In particular, $C$ is an $(\Ad(x)\circ \s)$-Coxeter conjugacy class in $\breve K$. 

\subsection{Affine Deligne--Lusztig varieties in the affine Grassmannian}\label{chen-zhu} It remains to show that $\breve K$ occurring in \S\ref{bound} is very special. To do this, we need some information on  affine Deligne--Lusztig varieties in the affine Grassmannian. 

Let $\breve P \subseteq \breve G$ be a special parahoric subgroup containing $\breve I$. The affine Deligne--Lusztig variety in the affine Grassmannian $\breve G /\breve P$ is defined by $$X_\mu(b) = \{g \in \breve G/\breve P; g\i b \s(g) \in \breve P t^\mu \breve P\}.$$ 

The following dimension formula is proved in \cite{GHKR06} and \cite{Vi06} for split groups, \cite{Ham15} and \cite{Zhu} for unramified groups, and \cite[Theorem 2.29]{He-CDM} in general. 

\begin{theorem}\label{thm:gr}
Suppose that $[b] \in B(\bG, \mu)$. Then $\dim X_\mu(b)=\<\mu-\nu_b, \rho\>-\frac{1}{2}\de(b)$. 
\end{theorem}

Let $\Sigma^{\mathrm{top}}(X_\mu(b))$ be the set of top-dimensional irreducible components of $X_\mu(b)$.

Let $\widehat{\bG}$ be the Langlands dual of $\bG$ over the complex number field $\BC$. Let $\widehat{T}$ be the maximal torus dual to $T$. Then $\s$ acts on $\widehat{T}$ in a natural way, and we denote by $\widehat{T}^{\s}$ the $\s$-fixed points of $\widehat{T}$. Let $\l_b \in X^*(\widehat{T}^{\s})$ be the ``best integral approximation'' of the Newton point of $b$ in the sense of \cite[Definition 2.1]{HV}. Let $V_\mu$ be the irreducible representation of $\widehat{\bG}$ with highest weight $\mu$. Write $V_\mu(\l_b)$ for the corresponding $\l_b$-weight subspace of $\widehat{T}^{\s}$. The following result was conjectured by M. Chen and X. Zhu, and is proved in \cite{ZZ}, \cite{HZZ}, and \cite{Nie}.
\begin{theorem}\label{Chen--Zhu} The number of $\bJ_b(F)$-orbits on $\Sigma^{\mathrm{top}}(X_\mu(b))$ equals $\dim V_\mu(\l_b)$.  Moreover, the stabilizer of each element in $\Sigma^{\mathrm{top}}(X_\mu(b))$ is a very special parahoric subgroup of $\bJ_b(F)$.
\end{theorem}

We also need the following result that connects affine Deligne--Lusztig varieties in the affine flag and in the affine Grassmannian. 

\begin{lemma} \label{embedding}
The $\bJ_b(F)$-equivariant projection map $X_{t^\mu c}(b) \to X_\mu(b)$ is injective. 
\end{lemma}
\begin{proof}
Let $g \breve I, g' \breve I \in \breve G /\breve I$ be in the same fiber of the natural projection map $X_{t^\mu c}(b) \to X_\mu(b)$. Then $g'^{-1} g \in \breve P$. We have $g'^{-1} g  \in \breve I \dot x \breve I$ for some $x \in W$. 

Since $(g'^{-1} g)(g^{-1} b \s(g)) = (g'^{-1} b\s(g'))\s(g'^{-1} g)$, then $(\breve I \dot x \breve I) (\breve I t^{\mu} \dot c \breve I) \cap (\breve I t^{\mu} \dot c \breve I) \s(\breve I \dot x \breve I) \ne \emptyset$. Since $t^\mu c \in {}^\BS \tW$, $(\breve I \dot x \breve I) (\breve I t^\mu \dot c \breve I)=\breve I \dot x t^\mu \dot c \breve I$. Thus we have $x t^\mu c=t^\mu c\s(x)$ and $\supp_{\s}(x) \subset I(t^\mu c)$. As $c$ is $\s$-elliptic, we conclude that $\supp_{\s}(x) = \emptyset$ and hence $g' g\i \in \breve I$ as desired.
\end{proof}

\subsection{Proof of Theorem~\ref{thm:main}(1) and (2)}
We continue our analysis of reduction paths. All the notation is the same as in \S\ref{bound}.  

Note that the equalities in \S\ref{bound}(a) and (b) hold. Note also that $\de(b) = \sharp(I(\nu_b)/\<\s\>) - \sharp (\breve K/\<\Ad(x)\circ \s\>)$. It follows that $\ell_I(\underline p) = \sharp(\BS/\<\s\>) - \sharp(I(\nu_b)/\<\s\>) = \ell_I(\mu,[b])$. Using \S\ref{sec:cordial}(a) and the simple fact that $\ell_I(\underline p)+2\ell_{II}(\underline p) = \ell(w)-\ell(e)$, one can prove that $\ell_{II}(\underline p)=\ell_{II}(\mu,[b])$. This proves part (1) of Theorem~\ref{thm:main}.

By \S\ref{sec:classicalDL}(a), the stabilizer in $\bJ_b(F)$ of any irreducible component of $X_{\underline{p}}$ is isomorphic to the parahoric subgroup $\breve P_{\breve K} \cap \bJ_b(F) \subseteq \bJ_b(F)$. By \S\ref{bound}, we have $\dim X_{\underline p} =d_w(b)= \dim X_{\mu}(b)$. By Lemma~\ref{embedding}, the image of each irreducible component $Z$ of $X_{\underline{p}}$ in $X_\mu(b)$ is an open dense subset of some top-dimensional irreducible component $Y$ of $X_\mu(b)$. Thus the stabilizers of $Z$ and $Y$ coincide. By Theorem~\ref{Chen--Zhu}, $\breve P_{\breve K} \cap \bJ_b(F)$ is a very special parahoric subgroup of $\bJ_b(F)$. Hence, by \S\ref{sec:special}(a), $\breve K \subset \tilde J$ is very special with respect to $\Ad(\dot x) \circ \s$. By~\ref{bound}, $C$ is an $(\Ad(x) \circ \s)$-Coxeter conjugacy class of $W_{\breve K}$. This proves part (2) of Theorem~\ref{thm:main}. 

\subsection{The extreme cases} \label{extreme-case}  Let $\CT$ be a reduction tree of $t^\mu c$. Let $[b] \in B(\bG)_{t^\mu c}$ and let $\underline p$ be a path in $\CT$ such that $[b]_{\underline p} = [b]$. 

If $[b] = [b_{\mu, \bG\text{-}\indec}]$, then $\ell_{II}(\underline p) = \ell_{II}(\mu,[b]) = 0$. Therefore $\underline p$  consists only of type I edges and is unique.

If $[b]$ is basic. Then $I(\nu_b) = \BS$ and $\ell_I(\underline p) = \ell_{I}(\mu,[b]) =0$. Therefore $\underline p$  consists only of type II edges and is also unique. 

This proves the ``moreover'' part of Theorem~\ref{main} for these two extreme cases. 

\section{Some combinatorial identities}\label{sec:red}

\subsection{Reduction to combinatorial identities}\label{sec:red11}
In this section, we assume that $\mu$ is essentially non-central. Let $c$ be a $\s$-Coxeter element of $W$ such that $t^\mu c \in {}^{\BS} \tW$. Let $\CT$ be a reduction tree of $t^\mu c$. 
For any $[b] \in B(\bG, \mu)_{\indec}$, let $n_{[b]}$ be the number of reduction paths $\underline p$ in $\CT$ with $[b]_{\underline p}=[b]$. By Theorem~\ref{main}(2), $n_{[b]} \ge 1$ for all $[b] \in B(\bG, \mu)_{\indec}$. By \S\ref{extreme-case}, $n_{[b]}=1$ if $[b]$ is either the minimal or the maximal element in $B(\bG, \mu)_{\indec}$. 

Combining Theorem~\ref{thm:main}(1) and (2) with Proposition~\ref{prop:id} and Lemma~\ref{class-poly}, we have $$\bq^{\<\mu, 2 \rho\>-\sharp(\BS/\<\s\>)}=\sum_{[b] \in B(\bG, \mu)_{\indec}} n_{[b]} (\bq-1)^{\ell_I(\mu, [b])} \bq^{\ell_{II}(\mu, [b])+\ell_{[b]}},$$ where $\ell_{[b]} = \<\nu_b,2\rho\>+\sharp(I(\nu_b)/\<\s\>)-\de(b)$ and  equals $\ell(\CO)$ for any $\s$-Coxeter class $\CO$ associated with $[b]$. 

Note that $(\bq-1)^a \bq^{a'} \in \BN[\bq-1]$ for all $a, a' \in \BN$. Thus, to show that $n_{[b]}=1$ for all $[b]$, it suffices to show that 
\[\tag{$\spade$}\sum_{[b] \in B(\bG, \mu)_{\indec}} (\bq-1)^{\ell_I(\mu, [b])} \bq^{\ell_{II}(\mu, [b])+\ell_{[b]}}=\bq^{\<\mu, 2 \rho\>-\sharp(\BS/\<\s\>)}.\]
(In fact, it is enough to prove the inequality $\ge$.)

Using \S\ref{sec:cordial}(a), one computes that
$$\ell_{II}(\mu,[b])+\ell_{[b]} - (\<\mu,2\rho\> - \sharp(\BS/\<\s\>)) = \sharp(I(\nu_b)/\<\s\>) -\leng ([b],[t^{\mu}]).$$
Thus, the equality $(\spade)$ is equivalent to 
\[\tag{$\spade'$} \sum_{[b] \in B(\bG, \mu)_{\indec}} (\bq-1)^{\sharp(\BS/\<\s\>) - \sharp(I(\nu_b)/\<\s\>)} \bq^{\sharp(I(\nu_b)/\<\s\>)- \mathrm{length}([b],[t^{\mu}])}=1.\]

\subsection{Reduction to unramified adjoint groups}
Let $\bG_{\ad}$ be the adjoint group of $\bG$ and let $T_{\ad}$ be the image of $T$ in $\bG_{\ad}$. We denote by $\mu_{\ad}$ the image of $\mu$ in $X_*(T_{\ad})_{\G_0}$. For any $b \in \breve G$, we denote by $b_{\ad}$ its image in $\breve G_{\ad}$. By \cite[Proposition 4.10]{kottwitz-isoII}, the map $\bG \to \bG_{\ad}$ identifies the reduced root system of $\bG_{\ad}$ with that of $\bG$ and induces an isomorphism of posets $$B(\bG, \mu)_\indec \cong B(\bG_{\ad}, \mu_{\ad})_\indec, \qquad [b] \mapsto [b_{\ad}].$$ Therefore, the equality ($\spade'$) for $\bG$ is equivalent to that for $\bG_{\ad}$. We can therefore assume that $\bG$ is adjoint. In this case, it is convenient to work with the reduced root system $\Phi$ of $\bG$. Define $$B(\Phi, \s, \mu) = \{\nu \in (V^+)^\s; \< \mu^\diamond - \nu,\o_{\CO_i}\> \in \BZ_{\ge 0} \text{ for any } i \in \BS - I(v)\},$$ where $\CO_i$ denotes the $\s$-orbit of $i$. By Lemma \cite[Lemma 3.5]{HN2}, the map $[b] \mapsto \nu_b$ identifies $B(\bG, \mu)$ with $B(\Phi, \s, \mu)$ as posets. For any $\nu\in B(\Phi,\s,\mu)$, we set $\leng(\nu, \mu^\diamond) = \sum_{\CO \in \BS/\<\s\>} \lceil \<\mu^\diamond - \nu,\o_\CO\> \rceil$. Then $\leng([b],[t^{\mu}]) = \leng(\nu_b,\mu^{\diamond})$ for any $[b]\in B(\bG,\mu)$.

We set 
$$ f_{\Phi,\s, \mu}(\nu) = (\bq-1)^{\sharp(\BS/\<\s\>) - \sharp(I(\nu)/\<\s\>)} \bq^{\sharp (I(\nu)/\<\s\>) - \leng(\nu, \mu^\diamond)}.$$
Now the equality ($\spade'$) can be reformulated in a purely combinatorial way as \[\tag{$\spade''$} \sum_{\nu \in B(\Phi,\s, \mu)_{\indec}} f_{\Phi,\s, \mu}(\nu) = 1.\] As any triple $(\Phi,\s, \mu)$ arises from an unramified group, it suffices to prove ($\spade''$) for the triples $(\Phi,\s, \mu)$ arising from unramified adjoint groups. In the rest of this section, we assume that $\bG$ is an unramified adjoint group. 

\subsection{Reduction to $F$-simple groups}
Write $\bG=\bG_1 \times \cdots \times \bG_l$, where $\bG_i$ are $F$-simple adjoint groups. Write $\mu=(\mu_1, \ldots, \mu_l)$, where $\mu_i$ is a dominant coweight of $\bG_i$. Also we have $\Phi = \Phi_1 \sqcup \cdots \sqcup \Phi_l$, where $\Phi_i$ is the root system of $\bG_i$. It is easy to see that \[B(\Phi,\s, \mu)_\indec = B(\Phi_1, \s_1, \mu_1)_\indec \times \cdots \times B(\Phi_l, \s_l, \mu_l)_\indec,\] where $\s_i$ is the restriction of $\s$ on $\Phi_i$. Moreover, it is easy to see that $f_{\Phi,\s, \mu} = \prod_{i=1}^l f_{\Phi_i,  \s_i,\mu_i}$. In the rest of this section, we assume that $\bG$ is an $F$-simple unramified  adjoint group. 

\subsection{Reduction to $\breve F$-simple groups} We have $\Phi=\Phi_1 \sqcup \cdots \sqcup \Phi_l$, where $\Phi_1 \cong \cdots \cong \Phi_l$ are irreducible root systems and $\s$ induces an isomorphism from $\Phi_i$ to $\Phi_{i+1}$. Here, by convention, we set $\Phi_{l+1}=\Phi_1$. Then the map \[\nu=(\nu_1, \dots, \nu_l) \mapsto |\nu| = \nu_l + \s(\nu_{l-1}) + \cdots+\s^{l-1}(\nu_1)\] induces an isomorphism of posets \[B(\Phi, \s, \mu)_\indec \overset \sim \to B(\Phi_l, \s^l, |\mu|)_\indec.\] Moreover, there is a natural bijection $\BS / \<\s\> \cong \BS_l / \<\s^l\>$, where $\BS_l$ is the set of simple reflections for $\Phi_l$. Thus $f_{\Phi, \s, \mu} = f_{\Phi_l, \s^l, |\mu|}$. In the rest of this section, we assume that $\bG$ is an $\breve F$-simple unramified  adjoint group.

\subsection{Reduction to split groups}Let $\CO$ be a $\s$-orbit of $\BS$. If all the simple roots in $\CO$ commute with each other, we define $\a_\CO' = \sum_{i \in \CO} \a_i$. If $\CO = \{i_0, j_0\}$ with $\< \a_{j_0}^\vee,\a_{i_0}\> = \< \a_{i_0}^\vee,\a_{j_0}\>=-1$, we define $\a_\CO' = 2(\a_{i_0} + \a_{j_0})$. Let $\BS'=\BS / \<\s\>$ and let $\Phi'$ be the root system generated by $\a_\CO'$ for all $\CO \in \BS'$. The coroot corresponding to $\CO$ is given by $\a_\CO'^{\vee} = \frac{1}{\sharp \CO} \sum_{i \in \CO} \a_i^\vee$, and the fundamental weight corresponding to $\CO$ is given by $\o_{\CO}'=\sum_{i \in \CO} \o_i $. For any $\nu \in (V^+)^\s$, $\nu \in B(\Phi, \s,\mu)$ if and only if $\<\mu^{\diamond} - \nu, \o_{\CO}'\> \in \BZ_{\ge 0}$ for any $\CO \in \BS/\<\s\>$ such that $\<\nu, \a_{\CO}'\> \neq 0$, which is also equivalent to $\nu \in B(\Phi',\id, \mu^{\diamond})$. Hence we have the following:

\smallskip

(a) {\it The natural identification $(\BR\Phi^\vee)^\s = \BR{\Phi'}^\vee$ induces an bijection of posets $$B(\Phi,  \s,\mu)_\indec \overset \sim \to B(\Phi',\id, \mu^{\diamond})_\indec.$$}

\smallskip

\noindent It follows from (a) that $f_{\Phi,\s, \mu } = f_{\Phi', \id, \mu^{\diamond}}$. Therefore, it suffices to prove ($\spade''$) for the triples $(\Phi, \s,\mu)$ arising from split groups. In the rest of this section, we assume that $\bG$ is a split  $\breve F$-simple  adjoint group. We identify $B(\bG,\mu)$ with $B(\Phi,\s,\mu)$ and write $f_{\bG,\mu}(\nu)= f_{\Phi,\s,\mu}(\nu)$.

\subsection{Reduction to simply laced groups} By \S\ref{sec:red}, the identity ($\spade$) is equivalent to the condition that in some (or,  equivalently, any) reduction tree of $t^\mu c$, there exists only one reduction path whose end point is associated with a given $[b] \in B(\bG, \mu)_{\indec}$. 

There exists an irreducible, simply laced, extended affine Weyl group $(\tW',\tilde\BS')$ of adjoint type and a length-preserving automorphism $\d$ on $\tW'$ such that $\tW=(\tW')^\d$. We have a natural bijection between $\tilde\BS$ and $\tilde{\BS}'/\<\d\>$. Moreover, we may assume that the simple reflections in each $\d$-orbit in $\tilde\BS'$ commute. More explicitly, 
\begin{itemize}
    \item if $\tW$ is of type $\tilde B_n$, then we take $\tW'$ to be of type $\tilde D_{n+1}$ and $\d$ is of order $2$;
    \item if $\tW$ is of type $\tilde C_n$, then we take $\tW'$ to be of type $\tilde A_{2 n-1}$ and $\d$ is of order $2$;
    \item if $\tW$ is of type $\tilde F_4$, then we take $\tW'$ to be of type $\tilde E_6$ and $\d$ is of order $2$;
    \item if $\tW$ is of type $\tilde G_2$, then we take $\tW'$ to be of type $D_4$ and $\d$ is of order $3$. 
\end{itemize}

Let $\iota: \tW \to \tW'$ be the natural embedding. For each $i\in \tilde\BS$, we have $\iota(s_i) = s_{i_1'} \cdots s_{i_k'}$, where $i_1', \ldots, i_k'$ are the $\d$-orbits of $i$ in $\tilde\BS'$. Let $w\in\tW$ and $w\lup s_iw$ be a type I reduction edge (see \S\ref{sec:tree}). Then one can construct a $k$-step reduction path $$\iota(w)\lup s_{i_k'} \iota(w)\lup s_{i_{k-1}'}s_{i_k'}\iota(w) \lup \cdots \lup s_{i_1'}\cdots s_{i_k'} \iota(w) = \iota(s_iw)$$ in $\tW'$. Similarly, a type II reduction edge $w\lup s_iws_i$ corresponds to a $k$-step reduction path from $\iota(w)$ to $\iota(sws)$, whose edges are all of type II. Now considering a reduction tree $\G$ of $w$, we can construct a reduction tree $\CT'$ of $\iota(w)$ such that $\CT$ can be viewed as a subtree of $\CT'$ in the above way. Hence the multiplicity-one result of $\iota(t^{\mu}c)\in \tW'$ implies the multiplicity-one result of $t^{\mu}c\in \tW$.

\medskip

In the rest of this section, we assume that $\bG$ is a split  $\breve F$-simple  simply laced adjoint group. We will then reduce to the case where $\mu$ is a fundamental coweight. We first need a combinatorial identity on finite graphs.  

\subsection{A combinatorial identity on graphs} \label{tree-subsec} Let $X$ be a finite graph and $Y \subseteq X$. Denote by $\CA(Y,X)$ the set of subsets $J\subseteq X$ such that none of the connected components of $X-J$ lies in $Y$. Define $$f_{Y, X} = \sum_{J \in \CA(Y, X)} (\bq-1)^{\sharp (J - Y \cap J^\circ)} \bq^{\sharp(Y \cap J^\circ)} \in \BZ[\bq],$$ where $J^\circ =  \{i \in J; \text{$i$ has no neighbors in $X - J$}\}$ is the interior of $J$.
\begin{lemma} \label{tree}
We have $f_{Y, X} = \bq^{\sharp X}$ for any $Y \subseteq X$.
\end{lemma}
\begin{proof}
Define
\[\a: \{(J, K); J \in \CA(Y, X), K \subseteq Y \cap J^\circ \} \to  \{\text{subsets of $X$}\}, ~ (J, K) \mapsto J-K.\]
We construct the inverse map $\b$ of $\a$ as follows. Let $H \subseteq X$. Let $C$ be the union of connected components of $X - H$ that  are contained in $Y$. Then we define $\b(H) = (H \sqcup C, C)$. By definition, $\a \circ \b = \id$. On the other hand, for any $J \in \CA(Y, X)$ and $K \subseteq Y \cap J^\circ$, $\a((J, K)) = J - K$. Moreover, $X - (J - K) = (X - J) \sqcup K$. As $K \subseteq Y \cap J^\circ$, $K$ and $X-J$ are not connected with each other. Hence $K$ is the union of connected components of $X - (J - K)$ contained in $Y$. Therefore $\b \circ \a((J, K)) = \b( J - K ) = ((J-K) \sqcup K,K)=(J,K)$ and hence $\b \circ \a = \id$. Therefore $\a$ is a bijection. Using the binomial expansion, we get
\begin{align*}
f_{Y, X} &= \sum_{J \in \CA(Y, X)} (\bq-1)^{\sharp (J - Y \cap J^\circ)} \bq^{\sharp(Y \cap J^\circ)} \\ &= \sum_{J \in \CA(Y, X)} (\bq-1)^{\sharp (J - Y \cap J^\circ)} \sum_{K \subseteq Y \cap J^\circ} (q-1)^{\sharp (Y \cap J^\circ - K)} \\ &= \sum_{J \in \CA(X, Y)} \sum_{K \subseteq Y \cap J^\circ} (\bq - 1)^{\sharp (J - K)} \\
&=\sum_{H \subseteq X} (\bq-1)^{\sharp H}=\bq^{\sharp X}.
\end{align*}
The lemma is proved. 
\end{proof}

\subsection{Reduction to fundamental coweights}\label{sec:6.8} Assume that $\mu$ is not a fundamental coweight. Then there exist $i, j \in \BS$ (here $i$ and $j$ are not necessarily distinct) such that $\mu - \o_i^\vee - \o_j^\vee$ is also dominant. Let $X$ be the (unique) path in the Dynkin diagram of $\BS$ with end points $i, j$. 

Let $\l = \mu - \sum_{k \in X} \a_k^\vee$ and let $\CA$ be the set of subsets $J \subseteq X$ such that $\l$ is non-central on each connected component of $\BS - J$ (equivalently, $B(\bM_{\BS-J}, \l)_{\irr} \neq \emptyset$). Note that $\l$ is non-central on each connected component of $\BS-X$ and that the connected components of $\BS-J$ that do not intersect $\BS-X$ are the connected components of $\{i \in X-J; i \text{ has no neighbors in } \BS-X\}$. 

Set $Y=\{i \in X \cap I(\l); i \text{ has no neighbors in } \BS-X\}$. Then it is easy to see that $\CA=\CA(Y, X)$ defined in \S\ref{tree-subsec}.

\begin{lemma}
We have the following: 
\begin{enumerate}
\item $\l$ is dominant; 
\item $B(\bG, \mu)_{\indec} = \sqcup_{J \in \CA} B(\bG, \l)_{(\BS - J)\text{-}\irr}$.
\end{enumerate}
\end{lemma}

\begin{proof}
Let $l\in\BS$. If $l \in \BS - X$; then $\<\l,\a_l\> \ge -\<\sum_{k \in X} \a_k^\vee,\a_l \> \ge 0$. If $l \in X - \{i, j\}$, then $\<\l,\a_l\> = -\< \sum_{k \in X} \a_k^\vee,\a_l\> = -\< \a_{k}^\vee + \a_l^\vee + \a_{k'}^\vee,\a_l\> \ge -(-1 + 2 -1) = 0$, where $k$ and $k'$ are the neighbors of $l$ in $X$. If $l \in \{i,j\}$ and $i\ne j$, then $\<\l,\a_l\> = \<\mu,\a_l \>-\< \sum_{k \in X} \a_k^\vee,\a_l\> \ge 1-1=0$. If $l=i=j$, then $\<\l,\a_l\> = \<\mu,\a_l \>-\< \sum_{k \in X} \a_k^\vee,\a_l\> \ge 2-2=0$. This proves part (1).

By definition, $B(\bG, \l)_{(\BS - J)\text{-}\irr} \subseteq B(\bG, \mu)_{\indec}$ for any $J \subseteq X$. Now we prove the other direction. Let $E = \{e_i; i \in \BS\}$ be as in the proof of Proposition~\ref{prop:max-b}. Then, for each $i \in \BS$, by the definition of $\l$, we have $$\<\l,\o_i\> \ge \max\{0, \<\mu,\o_i \> - 1\} = \<e_i,\o_i \>.$$ By \cite[Theorem 6.5]{Chai}, this implies that $\l \in C_{\ge E}$ and hence $\l \ge \min C_{\ge E} = [b]_{\bG\text{-}\indec}.$ 

For any $\nu \in  B(\bG, \mu)_{\indec}$, there exists a unique subset $J \subseteq X$ such that $\l - \nu \in \sum_{l \in \BS - J} \BR_{> 0} \a_l^\vee$. By Lemma~\ref{Newton}, $\nu \in B(\bG, \l)_{(\BS - J)\text{-}\irr}$. Part (2) is proved.
\end{proof}

\begin{proposition}
Suppose that the equality ($\spade''$) holds for all fundamental coweights. Then it holds for all dominant coweights.
\end{proposition}
\begin{proof}
We argue by induction on the semisimple rank of $\bG$ and the number $\<\rho, \mu\>$. 

Suppose $\mu$ is not fundamental. Let the notation be as in \S\ref{sec:6.8}. By Corollary~\ref{cor:irr}, we identify $B(\bG, \l)_{(\BS - J)\text{-}\irr}$ with $B(\bM_{\BS - J}, \mu)_{\irr}$ for $J \in \CA$. We show that 

\smallskip

(a) {\it For any $J \in \CA$ and $\nu \in B(\bM_{\BS - J}, \mu)_{\irr}$, we have $$f_{\bG, \mu}(\nu) = \bq^{-\sharp X} (\bq-1)^{\sharp (J - Y \cap J^\circ)} \bq^{\sharp Y \cap J^\circ} f_{\bM_{\BS - J}, \l}(\nu).$$}

\smallskip

By definition, $f_{\bM_K, \mu}(\nu) = (\bq-1)^{\sharp (K - K \cap I(\nu))} \bq^{\sharp K \cap I(\nu) - \leng _{\bM_K}(\mu, \nu)}$ for any $K \subseteq \BS$. 

Note that $\leng _\bG(\l, \mu) = \<\mu - \l,\rho\> = \sharp X$, and hence $$\leng _{\bG}(\nu, \mu) = \leng _{\bG}(\nu, \l) + \leng _{\bG}(\l, \mu) = \leng _{\bG}(\nu, \l) + \sharp X.$$ By Lemma~\ref{Newton}, we have $\leng _{\bG}(\nu, \mu) = \leng _{\bM}(\nu, \mu)$. To show (a), it remains to show that $I(\nu) - (\BS-J) \cap I(\nu) = J \cap I(\nu) = Y \cap J^\circ$. Let $l \in J$. Write $\nu = \l -\d$ for some $\d \in \sum_{k \in \BS - J} \BR_{>0} \a_k^\vee$. Then $l \in J \cap I(\nu)$ if and only if $\<\nu,\a\> = \<\l,\a_l\> - \<\d,\a_l\> = 0$, which is also equivalent to  $\<\l,\a_l\> = \<\d,\a_l\> = 0$, that is, $l \in Y \cap J^\circ$ as desired. (a) is proved. 

Now, by Lemma~\ref{tree} and the inductive hypothesis on $\bM_{\BS-J}$ and $\l$, we have 
\begin{align*} 
    &\sum_{\nu \in B(\bG, \mu)_{\indec}} f_{\bG, \mu}(\nu) =\sum_{J \in \CA} \sum_{\nu \in B(\bM_{\BS - J}, \l)_{\irr}}  f_{\bG, \mu}(\nu) \\ &\hspace{56pt}=\bq^{-\sharp X} \sum_{J \in \CA} \sum_{\nu \in B(\bM_{\BS - J}, \l)_{\irr}} (\bq-1)^{\sharp(J - Y \cap J^\circ)} \bq^{\sharp Y \cap J^\circ} f_{\bM_{\BS- J}, \l}(\nu) \\ &\hspace{56pt}= \bq^{-\sharp X} \sum_{J \in \CA} (\bq-1)^{\sharp(J - Y \cap J^\circ)} \bq^{\sharp Y \cap J^\circ}= 1.
\end{align*} The proof is finished.
\end{proof}

Now it is sufficient to deal with the fundamental coweights of $\bG$. 

\subsection{Proof for the minuscule coweights}\label{sec: minu}
Let $\mu$ a be (non-zero) minuscule coweight. Then $\dim V_{\mu}(\l_b) = 1$ for any $[b] \in B(\bG,\mu)$. By Theorem~\ref{Chen--Zhu}, we conclude that $\sharp (\bJ_b(F)\backslash\Sigma^{\mathrm{top}}(X_{\mu}(b))) = 1$ for any $[b] \in B(\bG,\mu)$. As in \S\ref{sec:red11}, we have
$$\bq^{\<\mu, 2 \rho\>-\sharp(\BS/\<\s\>)}=\sum_{[b] \in B(\bG, \mu)_{\indec}} n_{[b]} (\bq-1)^{\ell_I(\mu, [b])} \bq^{\ell_{II}(\mu, [b])+\ell_{[b]}},$$ where $n_{[b]}$ is the number of reduction paths $\underline p$ in a given reduction tree $\CT$ of $t^\mu c$ with $[b]_{\underline p}=[b]$. Note also that at the end of \S\ref{bound}, we showed that $\dim X_{\underline p}=\dim X_{t^{\mu}c}([b]_{\underline p}) = \dim X_{\mu}([b]_{\underline p})$ for any reduction path $\underline p$. Using Lemma~\ref{embedding}, we conclude that all $n_{[b]}=1$. This implies the combinatorial identity $(\spade)$, and then $(\spade'')$ follows.

In particular, the combinatorial identity ($\spade''$) holds for type $A$, since all the fundamental coweights are minuscule. For $(A_{n-1}, \o^\vee_i)$, we may write the identity $(\spade)$ explicitly as 
\begin{multline*}
\sum_{\substack{k \ge 1, 1>\frac{a_1}{b_1}>\cdots>\frac{a_k}{b_k}>0; \\ a_i+\cdots+a_k=i, b_1+\cdots+b_k=n}} (\bq-1)^{k-1} \bq^{k-1-\frac{\sum_{1 \le l_1<l_2\le k}(a_{l_1} b_{l_2}-a_{l_2} b_{l_1})+\sum_{1 \le l \le k} \gcd(a_l, b_l)}{2}}\\=\bq^{\frac{i(n-i)-n}{2}}.
\end{multline*}
We do not know if there is a purely combinatorial proof of this identity.

\subsection{Type $D_n$ case}\label{sec:D} In this subsection, we assume that $\bG$ is of type $D_n$ $(n\ge 4)$. Note that the fundamental coweights $\omega^{\vee}_1, \omega^{\vee}_{n-1},\omega^{\vee}_n$ are minuscule and have already been dealt with in \S\ref{sec: minu}. Here we  deal only with $\omega^{\vee}_i$ for $2\leq i \leq n-2$.

For any integer $k\ge2$, denote by $[1,k]$ the set $\{1,2,\ldots,k\}$. We have a natural embedding $B(\bG,\omega_{i-2}^{\vee})\rightarrow B(\bG,\omega_i^{\vee})$. Set
\begin{align*}
B_I& = \bigsqcup_{i-2\le k \le n-3}B(\bG,\omega^{\vee
}_{i-2})_{[1,k]\text{-}\irr}\cong \bigsqcup_{i-2\le k \le n-3}B(\bM_{[1,k]},\omega^{\vee
}_{i-2})_{\irr},\\
B_{II} &= \bigsqcup_{J\subseteq \{n-1,n\}}B(\bG,\omega^{\vee
}_{i-2})_{(\BS-J)\text{-}\irr}\cong \bigsqcup_{J\subseteq \{n-1,n\}}B(\bM_{\BS- J},\omega^{\vee
}_{i-2})_{\irr}.
\end{align*}

For $i \le k \le n-2$, the adjoint group of $\bM_{\BS-\{k\}}$ is of type $A_{k-1} \times D_{n-k}$. Here, by convention, type $D_3$ is the same as type $A_3$, and type $D_2$ is the same as type $A_1\times A_1$. Set $\mu_k=(1^{i-1}, 0^{k-i+1}, 1, 0^{n-k-1})$. Then $\mu$ and $\omega^\vee_i$ are in the same $W$-orbit. The restriction of $\mu_k$ to $\bM_{\BS-\{k\}}$ is (dominant) minuscule, and its projection to the adjoint group of $\bM_{\BS-\{k\}}$
is the coweight $(\o^\vee_{i-1}, \o^\vee_1)$ if $k<n-2$ and $(\o^\vee_{i-1}, \o^\vee_1,\o_1^{\vee})$ if $k=n-2$. As in \S\ref{sec:HN}, we identify $B(\bM_{\BS-\{k\}}, \mu_k)$ with its natural image in $B(\bG,\o_i)$. Set
$$B_{III} = \bigsqcup_{i\le k \le n-2}B(\bG, \mu_k)_{(\BS-\{k\})\text{-}\irr}\cong\bigsqcup_{i\le k \le n-2}B(\bM_{\BS-\{k\}}, \mu_k)_{\irr}.$$ 

By direct computation, $B_I$, $B_{II}$, and $B_{III}$ are disjoint in $B(\bG,\omega^{\vee}_i)$. In the rest of this section, we will show that
\[\tag{**} \sum_{\nu\in B_I\sqcup B_{II}\sqcup B_{III}}f_{\bG,\omega^{\vee}_{i}}(\nu) =  1.\]
In \S\ref{sec:red11}, we have already shown that the left-hand side of (**) is less than or equal to 1. Hence, if we prove (**), the equality $B_I\sqcup B_{I}\sqcup B_{III}=B(\bG,\o_i^{\vee})_{\indec}$ also follows.

It can be checked directly that (**) holds for $D_4$. Using induction, we may assume that (**) holds for groups of type $D$ with semisimple rank less than $n$. Note also that ($\spade''$) for type $A$ has already been proved in \S\ref{sec: minu}. Therefore we have 

\smallskip

(a) $\sum_{\nu\in B_{II}}f_{\bG,\omega^{\vee}_{i}}(\nu) =  \bq^{-\<\a^{\vee},\rho\> }\cdot(1+2(\bq-1)^2+(\bq-1)^2) = \bq^{-\<\a^{\vee},\rho\> }\cdot \bq^2$.

\smallskip

Next we handle $B_I$ and $B_{III}$. Let $$\a^{\vee} = \omega^{\vee}_{i}-\omega^{\vee}_{i-2}= \a_{i-1}^{\vee} +2\a_{i}^{\vee}+\dots +2\a_{n-2}^{\vee}+\a_{n-1}^{\vee}+\a_{n}^{\vee}.$$
Note that $\<\a^{\vee},\rho\> = 2n - 2i +1$.
We claim that

\smallskip

(b) {\it For $k \in [i-2,n-3]$ and $\nu \in B(\bM_{[1,k]},\omega^{\vee}_{i-2})_{\irr}$, we have
$$f_{\bG,\omega^{\vee}_{i}}(\nu) = (\bq-1)\cdot \bq^{-\<\a^{\vee},\rho\>+n-k-1}\cdot f_{\bG,\omega^{\vee}_{i-2}}(\nu).$$}

(c) {\it For $k \in [i,n-2]$ and $\nu \in B(\bM_{\BS-\{k\}},\mu_k)_{\irr}$, we have
$$f_{\bG,\omega^{\vee}_{i}}(\nu) = (\bq-1)\cdot \bq^{-\<\a^{\vee},\rho\> + 2n-k-i}\cdot f_{\bM_{\BS-\{k\}},  \mu_k}(\nu).$$}
We prove (c) here. The proof of (b) is similar.

By definition, we have
$$f_{\bG,\omega^{\vee}_{i}}(\nu) = (\bq-1)\cdot \bq^{n-k-n_0}\cdot f_{\bM_{\BS-\{k\}},\mu_k}(\nu),$$
where $n_0= \leng _{\bG}(\nu,\omega^{\vee}_i) - \leng _{\bM_{\bG-\{k\}}}(\nu,\mu_k) $. It can be checked directly that
$$\<\mu_k - \nu, \rho_{\bM_{\BS-\{k\}}}  \> = \<\o_{i-1}^{\vee}  - \nu, \rho   \> + \< (0^{k},1,0^{n-k-1}), \rho_{\bM_{[k+1,n]}}   \> =  \<\o_{i-1}^{\vee}  - \nu, \rho   \>+ 1.$$
Note also that $\de_{\bG}(\nu) = \de_{\bM_{\BS-\{k\}}}(\nu) $.
Therefore, by the length formula, we have
\begin{align*}
n_0 &= \<\omega^{\vee}_i -\nu,\rho \> - \<\mu_k - \nu, \rho_{\bM_{\BS-\{k\}}}  \>+\frac{1}{2}\left( \de_{D_n}(\nu)-\de_{\bM_{\BS-\{k\}}}(\nu) \right)\\
&= \<\o^{\vee}_i - \o_{i-1}^{\vee},\rho \>  + 1= n-i+ 1=\<\a^{\vee},\rho\>-n+i.
\end{align*}
The statement (c) follows.

Combining (a), (b), and (c) with the combinatorial identity $(\spade)$ for type $A$ and for type $D_l$ with $l<n$, we have
\begin{align*}
&\sum_{\nu\in B_I\sqcup B_{II}\sqcup B_{III}}f_{D_n,\omega^{\vee}_{i}}(\nu) \\
&\hspace{64pt}= \bq^{-\<\a^{\vee},\rho\>}\cdot \left(\bq^2 + (\bq-1)\!  \left( \sum_{k=i-2}^{n-3}\bq^{n-k-1}+\sum_{k=i}^{n-2}\bq^{2n-k-i} \right)\right)\\
&\hspace{64pt} = \bq^{-\<\a^{\vee},\rho\>}\cdot \bq^{2n-2i+1}=1.
\end{align*}
This completes the proof of (**).

\subsection{Type $E$ case}
In this subsection, we assume that $\bG$ is of type $E_n$. We verify the combinatorial identity $(\spade'')$ by computer. Recall that a vector $v\in (V^+)^{\s}$ lies in $B(\bG,\mu)$ if and only if  $\<\mu-v,\o_i\>\in \BZ_{\ge0}$ for any $i\in \BS - I(v)$. As a consequence, we have the following characterization of $B(\bG,\mu)_{\indec}$: 

\smallskip

(a) {\it 
The set $B(\bG,\mu)_{\indec}$ equals the set of dominant vectors of the form $\nu = pr_I(\mu - \sum_{i\in I}c_i\a^{\vee}_i)$, where $I$ is a $\s$-stable subset of $\BS$, $c_i\in \BZ$, $1\leq c_i\leq \<\mu,\omega_i\>$, and $pr_I:V \rightarrow \bigoplus_{i\in I}\BR \omega^{\vee}_i$ is the natural orthogonal projection.}

\smallskip

On the basis of (a), we can use a computer program to list all the elements in $B(\bG,\mu)_{\indec}$ and then verify the combinatorial identity $(\spade'')$ directly. 

In the following tables, we provide the numbers of elements in $B(\bG,\mu)_{\indec}$ for all the fundamental, non-minuscule coweights in type $E$. The most complicated case is $E_8$,  and $\mu = \omega^{\vee}_4$, in which $B(\bG,\mu)_{\indec}$ contains $729$ elements.

\begin{center}
\captionof{table}{Type $E_6$}
\begin{tabular}{|c|c|c|c|c|}
\hline
$\mu$ & $\omega_2^{\vee}$ & $\omega_3^{\vee}$ & $\omega_4^{\vee}$ &  $\omega_5^{\vee}$ \\
\hline
$\sharp B(\bG,\mu)_{\indec}$ & $7$ & $15$ & $30$ & $ 15$ \\
\hline
\end{tabular}
\end{center}

\begin{center}
\captionof{table}{Type $E_7$}
\begin{tabular}{|c|c|c|c|c|c|c|}
\hline
$\mu$ & $\omega_1^{\vee}$ & $\omega_2^{\vee}$ & $\omega_3^{\vee}$ & $\omega_4^{\vee}$ &  $\omega_5^{\vee}$ & $\omega_6^{\vee}$ \\
\hline
$\sharp B(\bG,\mu)_{\indec}$ & $13$ & $26$ & $50$ & $125$ & $69$ & $32$ \\
\hline
\end{tabular}
\end{center}

\begin{center}
\captionof{table}{Type $E_8$}
\begin{tabular}{|c|c|c|c|c|c|c|c|c|}
\hline
$\mu$ & $\omega_1^{\vee}$ & $\omega_2^{\vee}$ & $\omega_3^{\vee}$ & $\omega_4^{\vee}$ &  $\omega_5^{\vee}$ & $\omega_6^{\vee}$ & $\omega_7^{\vee}$ & $\omega_8^{\vee}$ \\
\hline
$\sharp B(\bG,\mu)_{\indec}$ & $56$ & $126$ & $254$ & $729$ & $424$ & $220$ & $94$ & $27$ \\
\hline
\end{tabular}
\end{center}

\section{The general case}\label{sec:final}

\subsection{Description of the reduction trees}
Let $w \in Wt^{\mu}W$ with finite partial $\s$-Coxeter part, that is, $\eta_\s(w)$ is a partial $\s$-Coxeter element. For any $J\in [J_0(w),J(w)]_{\mu}$ and $[b]\in B(\bG,\mu)_{J\text{-}\irr}$, we set 
\begin{align*} J^{\flat,w} &= \{i\in I(\mu^{\diamond})\cap (J(w)-J)\mid i\text{ commutes with } J  \}, \\
\ell_I(w,[b],J) &= \sharp(J(w)/\<\s\>) -\sharp(J^{\flat,w}/\<\s\>)-\sharp (I(\nu_b)\cap J /\<\s\>), \\ \ell_{II}(w,[b],J)&=\leng ([b],[t^{\mu}])- \sharp( J_0(w)/\<\s\>).
\end{align*}

By Lemma~\ref{Newton}, there exists a unique $\s$-conjugacy class $[b]_{\bM_{J}}\in B(\bM_{J},\mu)$ such that $[b]_{\bM_J} \subseteq [b]$. We similarly define $[b]_{\bM_{J^{\flat,w}\cup J}}\in B(\bM_{J^{\flat,w}\cup J},\mu)$. In this case, a $\s$-Coxeter element associated with $[b]_{\bM_{J^{\flat,w}\cup J}}$ is equal to the product of a $\s$-Coxeter element of $W_{J^{\flat,w}}$ and a $\s$-Coxeter element associated with $[b]_{\bM_J}$.

The main result of this section is the following description of the reduction tree of $w$. 

\begin{theorem}\label{7.1}Let $w\in Wt^{\mu}W$ with $\eta_\s(w)$ a partial $\s$-Coxeter element. Let $\CT$ be a reduction tree of $w$. Then, for any $J\in [J_0(w),J(w)]_{\mu}$ and $[b]\in B(\bG,\mu)_{J\text{-}\irr}$, there exists a unique reduction path $\underline p$ in $\CT$ with $[b]_{\underline p}=[b]$. Moreover, 
\begin{enumerate}
    \item $\ell_I(\underline p)=\ell_I(w, [b],J)$ and $\ell_{II}(\underline p)=\ell_{II}(w, [b],J)$;
    \item $\operatorname{end}(\underline p)$ is a $\s$-Coxeter element associated with $[b]_{\bM_{J^{\flat,w}\cup J}}$. 
\end{enumerate}
\end{theorem}

Combining Theorem~\ref{7.1} with Proposition~\ref{prop:dec} and Remark~\ref{fibration}, we obtain Theorem~\ref{main}(3) for $w$. 

\subsection{Strategy}
The strategy for proving Theorem~\ref{7.1} is very different from that adopted for the proof of Theorem~\ref{thm:main}. In the latter case, we used the Chen--Zhu conjecture and the dimension formula to determine the end points of the reduction trees of $t^\mu c$. However, such a method, when applied to general $w$, cannot determine the end points.

The approach we use here is as follows. We first apply the partial reduction method and the class polynomials for $t^\mu c$ to calculate the class polynomials for $w$. As we mentioned earlier,  class polynomials, in general, contains less information than  reduction trees. Fortunately, for the elements $w$ we consider here, by combining the information on the class polynomials and the estimates on the type I-edges, we obtain the required information for {\bf any} reduction tree. 




The information about the class polynomials we need is contained in the following equality on the $\s$-cocenter of the Iwahori--Hecke algebra $H$: 
\[\tag{$\diamondsuit$}T_w +[H,H]_{\s} = \sum_{\substack{J\in [J_0(w),J(w)]_{\mu}, \\ [b]\in  B(\bG,\mu)_{J\text{-}\irr}}}(\bq-1)^{\ell_I(w,[b],J)}\bq^{\ell_{II}(w,[b],J) }T_{\CO_{w,[b]}}+[H,H]_{\s},\] where ${\CO_{w,[b]}}$ is the $\s$-conjugacy class containing a $\s-$Coxeter element associated with $[b]_{\bM_{J^{\flat,w}\cup J}}$.

\subsection{Proof of the identity $(\diamondsuit)$}

\subsubsection{}\label{sec:7.3.1} We consider the case where $w=t^\mu c \in {}^{\BS} \tW$ for some partial $\s$-Coxeter element $c$ of $W$. Let $J = \supp_\s(c)$. In this case, $J_0(w) = J(w) = J$. 

If $J = \BS$, that is, $c$ is a (full) $\s$-Coxeter element, then $(\diamondsuit)$ follows from Lemma~\ref{class-poly} and the description of the reduction tree in Theorem~\ref{thm:main}. In other words, the class polynomial $F_{w, \CO}$ is given by
$$F_{w, \CO} =\begin{cases} (\bq-1)^{\ell_I(w,[b],J)}\bq^{\ell_{II}(w,[b],J)} & \text{if } \CO = \CO_{w, [b]} \text{ for some }[b] \in B(\bG, \mu)_{\irr}, \\  0 & \text{otherwise.} \end{cases}$$

Assume $J \subsetneq \BS$. 
It follows from \cite[Theorem 7.3]{HN15} that $F_{w, \CO} = \sum_{\CO^J \subseteq \CO} F_{w, \CO^J}^J$, where $\CO^J$ denotes a $\s$-conjugacy class of $\tW_J$ and $F_{w, \CO^J}^J$ denotes the corresponding class polynomial for $\bM_J$. For any $[b]\in B(\bG,\mu)_{J\text{-}\irr}$, we have $\CO_{w,[b]}\supset \CO_{w,[b]_{\bM_J}}$. Hence $F_{w, \CO_{w,[b]} }=F^J_{w, \CO_{w,[b]_{\bM_J}} } =(\bq-1)^{\ell_I(w,[b],J)}\bq^{\ell_{II}(w,[b],J) }$. And $F_{w, \CO}=0$ if $\CO\ne \CO_{w,[b]}$ for any 
$[b]\in B(\bG, \mu)_{J\text{-}\irr}$. This proves $(\diamondsuit)$.

\subsubsection{}\label{sec:7.3.2} We consider the case where $w=c_1t^{\mu}c_2$ for some partial $\s$-Coxeter elements $c_1,c_2$ of $W$ such that  $t^{\mu}c_2\in{}^{\BS}\tW$ and $c_1\in W_{I(t^{\mu}c_2)}$. 

Set $J_1 = \supp_{\s}(c_1)$ and $J_2 = \supp_{\s}(c_2)$. One can construct a reduction tree $\CT$ of $t^\mu c_2$ in $\tW_{J_2}$ such that $c_1$ commutes with all the vertices in $\CT$. In particular, if $w_1 \lup w_2$ is one edge in $\CT$, then we have $c_1 w_1 \lup c_1 w_2$ in $\tW_{J}$. Note that $c_1$ is a minimal length element in its $\s$-conjugacy class in $\tW_{J_1}$, and the simple reflections in $\tW_{J_1}$ commute with the simple reflections in $\tW_{J_2}$. It is easy to see that if $w_1$ is an end point in $\CT$ (and hence is a minimal length element in its $\s$-conjugacy class in $\tW_{J_2}$ and commutes with $c_1$), then $c_1 w_1$ is a minimal length element in its $\s$-conjugacy class in $\tW_J$. 

Let $\CT'$ be the tree with vertices $c_1 w_1$ for all vertices $w_1 \in \CT$ and the edges $c_1 w_1 \lup c_1 w_2$ for all edges $w_1 \lup w_2$ in $\CT$. Then, from the above discussion, $\CT'$ is a reduction tree of $c_1 t^\mu c_2$ in $\tW_J$. Note that for $J,[b]$ as in Theorem~\ref{7.1}, we have $\ell_I(w,[b],J) = \ell_I(t^{\mu}c_2,[b],J) $ and $\ell_{II}(w,[b]) = \ell_{II}(t^{\mu}c_2,[b])$. Hence the identity $(\diamondsuit)$ for $w$ follows from the identity for $t^{\mu}c_2$ established in \S\ref{sec:7.3.1}.


\subsubsection{} We consider the case where $w=c_1t^{\mu}c_2$ for partial $\s$-Coxeter elements $c_1$ and $c_2$ in $W$ such that $t^{\mu}c_2\in{}^{\BS}\tW$ and $\supp_{\s}(c_1)\cap\supp_{\s}(c_2)=\emptyset$. 

We prove the identity $(\diamondsuit)$ by induction on $\ell(c_1)$. The case $\ell(c_1)=0$ is proved in \S\ref{sec:7.3.1}. 

Assume that $\ell(c_1)>0$. Let $i\in\BS$ such that $s_ic_1<c_1$. There are two cases as follows.

\medskip

Case (1): $c_2\s(s_i)\in{}^{I(\mu)}W$. Write $w_1=s_iw$ and $w_2=s_iw\s(s_i)=s_ic_1t^{\mu}c_2\s(s_i)$. Then $T_w+[H,H]_{\s} = (\bq-1)T_{w_1} +\bq T_{w_2} + [H,H]_{\s}$. Note that $J(w_2)=J(w)$, $J_0(w_1)=J_0(w)$, $J(w_1)=J(w)-\{\s^{\ell}(s_i)\mid\ell\in\BZ\}$, and $J_0(w_2)=J_0(w)\sqcup \{\s^{\ell}(s_i)\mid\ell\in\BZ\}$. 
Then $$[J_0(w),J(w)]_{\mu}=[J_0(w_1),J(w_1)]_{\mu}\sqcup[J_0(w_2),J(w_2)]_{\mu}.$$
By the induction hypothesis, it suffices to prove that 
\begin{enumerate}
\item[(a)] if $J\in[J_0(w_1),J(w_1)]_{\mu} $, then $J^{\flat,w}=J^{\flat,w_1}$; and 
\item[(b)] if $J\in[J_0(w_2),J(w_2)]_{\mu} $, then $J^{\flat,w}=J^{\flat,w_2}$. 
\end{enumerate}
Statement (b) is obvious since $J(w_2)=J(w)$. Let us prove (a). Suppose $J^{\flat,w}\ne J^{\flat,w_1}$; then $i \in I(\mu^{\diamond})$ and $s_{i}$ commutes with $J$. Since $J\supseteq J_0(w)$, $s_{i}$ also commutes with $J_0(w)$. Then $\mu$ is not essentially non-central over $J_0(w_2) = J_0(w)\sqcup \{\s^l(i)\mid l\in\BZ\}$, which is a contradiction. This completes the proof.

\medskip

Case (2): $\s(i)\in I(\mu)$, and $\s(s_i)$ commutes with $c_2$. Then the equality $(\diamondsuit)$ holds for $w$ if and only if it holds for $s_iw\s(s_i)$. We continue with the procedure until Case (1) happens. If Case (2) happens all the time and the procedure does not stop, then $c_1\in W_{I(t^{\mu}c_2)}$, and the equality $(\diamondsuit)$ follows from \S\ref{sec:7.3.2}.

\subsection{Proof of Theorem~\ref{7.1}}
\subsubsection{} We consider the case $w=c t^\mu$, where $c$ is a partial $\s$-Coxeter element of $W$. Let $\CT$ be a reduction tree of $w$. By Lemma~\ref{class-poly} and the identity $(\diamondsuit)$ for $w$, we have that for any $\s$-conjugacy class $\CO$ of $\tW$, 
\begin{multline*}
\sum_{\underline p; \operatorname{end}(\underline p) \in \CO} (\bq-1)^{\ell_I(\underline p)}\bq^{\ell_{II}(\underline p)}\\=\begin{cases} (\bq-1)^{\ell_I(w,[b],J)}\bq^{\ell_{II}(w,[b],J)} & \text{if } \CO=\CO_{w, [b]} \text{ for some } [b], \\ 0 & \text{otherwise}.\end{cases}\end{multline*}

Let $\underline p$ be a path in $\CT$. Set $e=\operatorname{end}(\underline p)$ and $[b] = [b]_{\underline p}$. Assume that $[b]\in B(\bG,\mu)_{J\text{-}\irr}$ for some $J$. As $(\bq-1)^{\ell_I(\underline p)}\bq^{\ell_{II}(\underline p)} \in \BN[\bq-1]$, there is no cancellation involved in the left-hand side of the above equality. Therefore, $e$ must be contained in $\CO_{w,[b]}$. 

As in \S\ref{bound}, we have \begin{align*} \ell_I(\underline p) & \ge \dim V_e -\dim V_w = \sharp(J(w)/\<\s\>) -\sharp(J^{\flat,w}/\<\s\>)-\sharp (I(\nu_b)\cap J /\<\s\>) \\ &=\ell_I(w,[b],J).\end{align*} 
Note that $\ell_I(\underline p)+2 \ell_{II}(\underline p)=\ell_I(w,[b],J)+2 \ell_{II}(w,[b],J)$. Thus $\deg_{\bq} (\bq-1)^{\ell_I(\underline p)}\bq^{\ell_{II}(\underline p)} \ge \deg_{\bq} (\bq-1)^{\ell_I(w,[b],J)}\bq^{\ell_{II}(w,[b],J)}$, with equality holding if and only if $\ell_I(\underline p) = \ell_I(w,[b],J)$. Again, since there is no cancellation involved in $\sum_{\underline p; \operatorname{end}(\underline p) \in \CO} (\bq-1)^{\ell_I(\underline p)}\bq^{\ell_{II}(\underline p)}$, we must have $\ell_I(\underline p) = \ell_I(w,[b],J)$. In this case, $(\bq-1)^{\ell_I(\underline p)}\bq^{\ell_{II}(\underline p)}=(\bq-1)^{\ell_I(w,[b],J)}\bq^{\ell_{II}(w,[b],J)}$. This also shows that for each $\CO=\CO_{w, [b]}$, there is a unique reduction path $\underline p$ with $\operatorname{end}(\underline p) \in \CO$. 

This completes the proof of Theorem~\ref{7.1} for $w=c t^\mu$. 


\subsubsection{}
Now we consider the general case. Let $w=xt^{\mu}y$ with $t^{\mu}y\in{}^{\BS}\tW$. Set $c=\s^{-1}(y)x$ and $w' =ct^{\mu}$. We relate $w$ and $w'$ as in the proof of \cite[Theorem 10.3]{He14}. Let $y=s_1s_2\cdots s_r$ be a reduced expression. Let $w^{(0)} = w'$, $w^{(1)} = s_1w^{(0)}\s(s_1)$, $w^{(2)} = s_2w^{(1)}\s(s_2),\ldots,w^{(r)}=w$. We have $w^{(0)}\rightarrow_{\s} w^{(1)}\rightarrow_{\s}\cdots \rightarrow_{\s}w^{(r)}$. This give a path $w'\lup w$, consisting of $\frac{1}{2}(\ell(w')-\ell(w))$ type II edges. We denote this path by $\underline p_0$. 

Let $\CT$ be a reduction tree of $w$. One may construct a reduction tree $\CT'$ of $w'$ containing  the concatenation $\underline p_0 \circ \CT$ as a subgraph. In particular, for any reduction path $\underline p$ in $\CT$, the concatenation $\underline p':=\underline p_0 \circ \underline p$ is a reduction path in $\CT'$. By definition, $\ell_I(\underline p) =\ell_I(\underline p')$ and $\ell_{II}(\underline p) +\frac{1}{2}(\ell(w')-\ell(w))=\ell_{II}(\underline  p')$. It is obvious that $J(w)=J(w')$ and $J_0(w')=\emptyset$. Hence $\ell_I(w,[b],J) = \ell_I(w',[b],J)$. On the other hand, we have $\sharp(J_0(w)/\<\s\>) = \leng ([b_w],[t^{\mu}]) = \frac{1}{2}(\ell(\eta_{\s}(w)+\ell(y)-\ell(x))) = \frac{1}{2}(\ell(w')-\ell(w))$,  where the second equality follows from the definition of cordial elements. Hence $\ell_{II}(w,[b],J)+ \frac{1}{2}(\ell(w')-\ell(w))= \ell_{II}(w',[b],J)  $. The statements for $\CT$ can now be deduced from the statements for $\CT'$.

\end{document}